\documentclass[11pt,letterpaper]{article}

\usepackage[utf8]{inputenc}
\usepackage{comment}
\usepackage{microtype}
\usepackage{graphicx}
\usepackage{subfig}
\usepackage{booktabs} % for professional tables
\usepackage{scalerel}
\usepackage[margin=1in]{geometry}

\usepackage{amssymb}
\usepackage{amsthm}
\usepackage{amsmath}
\usepackage{mathtools}

\usepackage{nicefrac}

\usepackage{algorithmic}
\usepackage{algorithm}
\usepackage[shortlabels]{enumitem}

\usepackage{url}
\usepackage{float}

\usepackage{pifont}
\usepackage{hhline}
\usepackage{multirow}

\usepackage{graphicx,wrapfig}

\def\smskip{\smallskip}

\def\texitem#1{\par\smskip\noindent\hangindent 25pt
               \hbox to 25pt {\hss #1 ~}\ignorespaces}

\newcommand{\BEAS}{\begin{eqnarray*}}
\newcommand{\EEAS}{\end{eqnarray*}}
\newcommand{\BEA}{\begin{eqnarray}}
\newcommand{\EEA}{\end{eqnarray}}
\newcommand{\BEQ}{\begin{eqnarray}}
\newcommand{\EEQ}{\end{eqnarray}}
\newcommand{\BIT}{\begin{itemize}}
\newcommand{\EIT}{\end{itemize}}
\newcommand{\BNUM}{\begin{enumerate}}
\newcommand{\ENUM}{\end{enumerate}}

\newcommand{\BA}{\begin{array}}
\newcommand{\EA}{\end{array}}

\newcommand{\reals}{\mathbb{R}}

\DeclareMathOperator*{\argmin}{\arg\!\min}

\newif\ifpagenumbering
\pagenumberingtrue

\pagenumberingfalse

\newsavebox{\theorembox}
\newsavebox{\lemmabox}
\newsavebox{\defnbox}
\newsavebox{\corollarybox}
\newsavebox{\remarkbox}
\newsavebox{\assbox}
\savebox{\theorembox}{\noindent\bf Theorem}
\savebox{\lemmabox}{\noindent\bf Lemma}
\savebox{\defnbox}{\noindent\bf Definition}
\savebox{\corollarybox}{\noindent\bf Corollary}
\savebox{\remarkbox}{\noindent\bf Remark}

\newtheorem{assumption}{Assumption}
\newtheorem{definition}{Definition}
\newtheorem{theorem}{Theorem}
\newtheorem{lemma}[theorem]{Lemma}

\newtheorem{corollary}[theorem]{Corollary}
\theoremstyle{remark}
\newtheorem{remark}{Remark}

\numberwithin{assumption}{section}
\numberwithin{definition}{section}
\numberwithin{theorem}{section}
\numberwithin{remark}{section}

\usepackage{hyperref}
\hypersetup{colorlinks,citecolor=blue,linktocpage,breaklinks=true}

\usepackage{cleveref}
\usepackage{appendix}

\usepackage[
backend=bibtex,
style=numeric-comp,
sorting=none,
maxbibnames=12,
]{biblatex}
\title{\bfseries\sffamily
Perturbed Gradient Descent via Convex Quadratic Approximation for Nonconvex Bilevel Optimization}

\author{
    \normalsize Nazanin Abolfazli$^*$, Sina Sharifi$^\dag$, Mahyar Fazlyab$^\dag$, Erfan Yazdandoost Hamedani$^*$
}
\begin{document}
\maketitle

%%%%%%%%%%%%%%%%%%%%%%%%%%%%%%%%%%%%%%%%%%%%%%%%%%%%%%%%%%%%%%%%%%%%%%%%%%%%%%%%%
\begin{abstract}
Bilevel optimization is a fundamental tool in hierarchical decision-making and has been widely applied to machine learning tasks such as hyperparameter tuning, meta-learning, and continual learning. While significant progress has been made in bilevel optimization, existing methods predominantly focus on the {nonconvex-strongly convex, or the} nonconvex-PL settings, leaving the more general nonconvex-nonconvex framework underexplored. In this paper, we address this gap by developing an efficient gradient-based method inspired by the recently proposed Relaxed Gradient Flow (RXGF) framework with a continuous-time dynamic \cite{sharifi2025safe}. In particular, we introduce a discretized variant of RXGF and formulate convex quadratic program subproblems with closed-form solutions. We provide a rigorous convergence analysis, demonstrating that under the existence of a KKT point and a regularity assumption {(lower-level gradient PL assumption)}, our method achieves an iteration complexity of \( \mathcal{O}(1/\epsilon^{1.5}) \) in terms of the squared norm of the KKT residual for the reformulated problem. Moreover, even in the absence of the regularity assumption, we establish an iteration complexity of \( \mathcal{O}(1/\epsilon^{3}) \) for the same metric.
% Bilevel optimization is a fundamental component of hierarchical decision-making, where one problem is nested within the constraints of another. 
% In this paper, we introduce a novel gradient-based method to solve bilevel optimization. Our approach begins by using gradient descent to minimize the upper-level objective function without considering the lower-level problem. 
% We then design a convex Quadratic Program (QP) that minimally perturbs the gradient descent directions to reduce the sub-optimality of the condition imposed by the lower-level problem. 
% We first assume a certain regularity condition holds, and guarantee finding a stationary point with $\mathcal{O}(1/k^{2/3})$. 
% We then relax the assumption and prove that the resulting iterative algorithms find a stationary point in $\mathcal{O}(1/k^{1/3})$.
Through extensive numerical experiments on convex and nonconvex synthetic benchmarks and a hyper-data cleaning task, we illustrate the efficiency and scalability of our approach.  
\end{abstract}

\makeatletter
% use symbol series for footnote marks and for hyperref anchors
\renewcommand{\thefootnote}{\fnsymbol{footnote}}
\renewcommand{\theHfootnote}{\fnsymbol{footnote}}
\makeatother

% ---------- footnotes ----------
\footnotetext[1]{%
  N.\ Abolfazli and E.\ Yazdandoost~Hamedani are with the
  Dept.\ of Systems \& Industrial Engineering, University of Arizona,
  Tucson, AZ 85721, USA.\\
  \ttfamily\small \{nazaninabolfazli, erfany\}@arizona.edu%
}

\footnotetext[2]{%
  S.\ Sharifi and M.\ Fazlyab are with the Dept.\ of Electrical \& Computer
  Engineering, Johns Hopkins University, Baltimore, MD 21218, USA.\\
  \ttfamily\small \{sshari12, mahyarfazlyab\}@jhu.edu%
}

\section{Introduction}
Bilevel optimization is a fundamental framework in hierarchical decision-making, in which one optimization problem is nested within another. This class of problem finds a plethora of applications in engineering \cite{pandvzic2018investments}, economics \cite{von1952theory}, transportation \cite{sharma2015iterative}, and machine learning \cite{finn2017model, rajeswaran2019meta,fei2006one,hong2020two,bengio2000gradient,hao2024bilevel,zhang2022revisiting}. 
A broad category of bilevel optimization problems can be written as optimization problems of the following form, also known as optimistic bilevel optimization. 
% \ssh{I think in this case, $\ell(x)$ is not well defined!}
\begin{align}\label{eq: bilevel task 1}\tag{BLO}
    &\min_{x \in \mathbb R^n,y\in\mathcal Y^\star(x)}  f(x,y) \\
    &\text{s.t.} \quad  \mathcal Y^\star(x) \! = \! \argmin_{y \in \mathbb R^m} ~g(x,y) \nonumber,%\notag \nonumber
\end{align}
where $f,g: \mathbb R^n\times\mathbb R^m \to \mathbb R$ represent the upper-level and lower-level objective functions. 
The implicit objective approach \cite{dempe2002foundations} reformulated \eqref{eq: bilevel task 1} by minimizing the problem defined below 
\begin{equation}
    \min_{x \in \reals^n} \ \ell(x) \quad \text{where} \quad \ \ell(x) \! := \min_{y \in \mathcal Y^\star(x)} f(x,y), \notag
    \end{equation}
where $\ell:\mathbb R^n\to \mathbb R$ denotes the implicit objective function.
Bilevel optimization problems are inherently non-convex and computationally challenging, largely due to the complex interdependence between the upper-level and lower-level problems.
% 

% \ssh{This paragraph is in part redundant comparing to section 1.1 and the rest could be merged with the literature review sec 1.2.}
The majority of recent works in bilevel optimization concentrate on scenarios where the lower-level problem is strongly convex. 
Under the strong convexity of $g(x,\cdot)$, the solution of the lower-level problem $\mathcal Y^\star(x)$ is a singleton, and bilevel optimization reduces to minimizing $\ell(x)$ whose gradient can be calculated with implicit gradient method \cite{pedregosa2016hyperparameter, gould2016differentiating, ghadimi2018approximation, lorraine2020optimizing, ji2021bilevel,li2022fully,abolfazli2023inexact}.
% \todo{add more papers for strongly convex including our FW.}
% Although the strong convexity assumption for
% lower-level guarantees a unique solution and results in a straightforward loss landscape for the lower-level problem, it is quite restrictive, limiting the applicability of bilevel optimization to many interesting areas.
Recently, to relax the strong convexity assumption in the lower-level problem, some works focused on bilevel optimization with merely convex lower-level objectives. However, this introduces significant challenges, particularly the presence of multiple lower-level local optima (i.e., a non-singleton solution set) may hinder the adoption of implicit-based approaches that rely on implicit function theorem \cite{sow2022constrained,liu2023averaged,lu2024first}. Bilevel optimization problems with nonconvex lower-level objectives are common in various machine learning applications, such as hyperparameter optimization in deep neural network training \cite{vicol2022implicit}, continual learning \cite{borsos2020coresets, hao2024bilevel}, and more. However, the above bilevel optimization methods primarily rely on the assumption of lower-level strong convexity or convexity, which significantly limits their effectiveness in handling non-convex lower-level problems. Recently, some works studied bilevel optimization with non-convex lower-level problems  \cite{liu2022bome,chen2023bilevel,huang2023momentum,huang2024optimal}. 
% Bilevel optimization problems with non-convex lower-level objectives are commonly encountered in machine learning tasks, particularly in hyperparameter optimization for training deep neural networks. \cite{vicol2022implicit}. 
% 
 {Detailed review of recent works closely related to ours is provided in \cref{sec:related_works}}.
Next, we provide an interesting machine learning application that can be framed as bilevel optimization problems.

\textbf{Adversarial Training:} Adversarial Training (AT) is a defense strategy that strengthens machine learning models against adversarial attacks by exposing them to adversarial examples during training. It originally was formulated as a min-max optimization problem where the defender (model trainer) minimizes the worst-case training loss, while the attacker generates adversarial perturbations to maximize this loss \cite{goodfellow2014explaining,madry2018towards}. Recently, \cite{zhang2022revisiting} formulated AT as a bilevel optimization problem, where the upper-level minimizes the training loss, and the lower-level generates adversarial perturbations.
\begin{align*}
    &\min_{\theta}~ \mathbb{E}_{\xi \sim \mathcal{D}} \left[ \ell_{\text{tr}} (\theta, x + \delta^*(\theta), \xi) \right] \\
    &\text{s.t.} \quad \delta^*(\theta) \in  \argmin_{\delta} \ell_{\text{atk}} (\theta, \delta; \xi),
\end{align*}
where the function $\ell_{\text{tr}}$ represents the training loss, and $\ell_{\text{atk}}$ as the objective for generating adversarial perturbations. Note that when the attack objective is chosen as $\ell_{\text{atk}} = -\ell_{\text{tr}}$, the bilevel optimization formulation becomes equivalent to the standard min-max adversarial training problem.

\subsection{Existing Approaches for Solving Bilevel Optimization Problems}
In this subsection, we provide a brief overview of traditional approaches for solving bilevel optimization problems. \\

\textbf{Hyper-gradient Descent:} Assume that the minimum of $g(x,\cdot)$ is unique for all $x$, resulting in the optimal solution map \( \mathcal Y^\star(x) \) of the lower-level problem being single-valued and
continuously differentiable. The most straightforward approach to solving \eqref{eq: bilevel task 1} is to perform a gradient descent on the implicit objective function $\ell(x)$. 
% (defined in \eqref{eq: bilevel task 1})
\[
\nabla \ell(x) = \nabla_x f(x, y^\star(x)) + D_x y^\star(x) \nabla_y f(x, y^\star(x))
\]
If the Hessian matrix \( \nabla_{yy}^2 g(x, y^\star(x)) \) is invertible, the implicit function theorem \cite{Rudin1976} guarantees that the map \( x \mapsto y^\star(x) \) is continuously differentiable. Moreover, the Jacobian of the solution map, \( D_x y^\star(x) \), can be derived by differentiating the implicit equation with respect to \( x \). 
\[
 {\nabla_{yx}^2 g(x, y^\star(x)) + D_x y^\star(x)\nabla_{yy}^2 g(x, y^\star(x)) = 0}
\]
% This results in the linear system 
% \begin{equation}\label{eq:linear-system}
% \nabla_{yx}^2 g(x, y^\star(x)) + \nabla_{yy}^2 g(x, y^\star(x)) D_x y^\star(x) = 0
% \end{equation}
% % Now, we can solve for $D_x y^\star(x)$ and obtain a gradient update rule on $x$. 
% Using this equation, the gradient of the implicit objective function $\ell(\cdot)$, also known as the hyper-gradient, can be calculated as follows,
% \begin{subequations}\label{eq:grad-implicit}
% \begin{align}
%     &\nabla \ell(x) = \nabla_x f(x,y^\star(x)) - \nabla_{yx}^2 g(x,y^\star(x))^\top v^\star(x), \\
%     &\nabla_{yy}^2 g(x,y^\star(x)) v^\star(x) = \nabla_y f(x,y^\star(x)).
% \end{align}
% \end{subequations}
This approach is sometimes known as the hyper-gradient descent. However, hyper-gradient descent is computationally
expensive due to the need to compute $D_x y^*(x)$ at each step. To alleviate the computational burden, various approximation methods have been developed to avoid direct inversion of the Hessian \cite{pedregosa2016hyperparameter,rajeswaran2019meta,grazzi2020iteration,ghadimi2018approximation,lorraine2020optimizing}. 
% To avoid this, it is common to consider the following surrogate map to estimate \eqref{eq:grad-implicit} by replacing the optimal lower-level solution with an approximated solution $y$,
% \begin{subequations}\label{eq:grad-surrogate}
% \begin{align}
%     &F(x,y) :=  \nabla_x f(x,y) - \nabla_{yx}^2 g(x,y)^\top v(x,y),\\
%     &\nabla_{yy}^2 g(x,y) v(x,y) = \nabla_y f(x,y).
% \end{align}
% \end{subequations}
% %
% Under suitable assumptions on the upper-level objective function $f$, this estimate can be controlled by the distance between the optimal solution $y^\star(x)$ and the approximated solution $y$.
% The main computational bottleneck is to solve the linear equation in \eqref{eq:linear-system}. To alleviate the computational burden, various approximation methods have been developed. 
% One approach is to solve the linear equation approximately using iterative methods, such as conjugate gradient \cite{}, Neumann series expansion \cite{}, or related techniques \cite{}. These methods avoid direct inversion of the Hessian, but still require solving the linear system at every iteration. Another
% popular approximation approach is to replace $D_x y^\star(x)$
% with K-th iteration of gradient descent or other optimization steps on $g(x,y)$ w.r.t y starting from certain initializations\cite{}.
In addition, the exact computation of $y^\star(x)$ can be mitigated by considering a common surrogate map through replacing the optimal lower-level solution with an approximated solution $y$,
\begin{subequations}\label{eq:grad-surrogate}
\begin{align}
    &F(x,y) :=  \nabla_x f(x,y) - \nabla_{yx}^2 g(x,y)^\top v(x,y),\\
    &\nabla_{yy}^2 g(x,y) v(x,y) = \nabla_y f(x,y).
\end{align}
\end{subequations}
Under suitable assumptions on the upper-level objective function $f$, this estimate can be controlled by the distance between the optimal solution $y^\star(x)$ and the approximated solution $y$. In \eqref{eq:grad-surrogate} the effect of Hessian inversion is presented in a separate term $v(x,y)$ which solves a parametric quadratic problem and can be approximated using one or multiple steps of gradient descent \cite{abolfazli2023inexact,li2022fully,arbel2021amortized}.%\todo{cite more papers}

\textbf{Value Function Approach:} Another approach is based on the observation that \eqref{eq: bilevel task 1} is equivalent to the following constrained optimization:
\begin{alignat}{2}\label{eq: bilevel task 1 c}
    & \min_{x,y} \ f(x,y) \quad 
    &\text{s.t.} \quad    g(x,y) - g^*(x) \leq 0\
\end{alignat}
where $g^*(x) = \min_{y} g(x,y)$.  Compared to hyper-gradient approaches, this method does not require computing the implicit derivative $D_x y^\star(x)$. %Although this eliminates the need for Hessian computation, it requires solving the lower-level problem at each iteration \cite{}. 
As shown in \cite{ye1995optimality} one cannot establish a KKT condition of \eqref{eq: bilevel task 1} through this reformulation since none of the standard constraint qualifications (Slater's, LICQ, MFCQ, CRCQ) hold for this reformulation. Moreover, \cite{xiao2023generalized} showed that the calmness condition which is the weakest constraint qualification does not hold for bilevel optimization with non-convex lower-level when employing the value function-based reformulation.
%\todo[inline]{The issue with this approach in nonconvex setting is that no constraint qualification may hold. }

\textbf{Stationary-Seeking Methods:} An alternative method is to replace the lower-level problem in \eqref{eq: bilevel task 1} with the stationarity condition. As a result, the bilevel optimization problem can be reformulated as the following constrained, non-convex single-level optimization problem:
\begin{alignat}{2}\label{eq: bilevel task 1 b}
    & \min_{x,y} \ f(x,y) \quad 
    &\text{s.t.} \quad   \nabla_y g(x,y)=0\
\end{alignat}
%Algorithms for nonlinear (possibly nonconvex) equality constrained optimization can then be applied \cite{}. 
The reformulated problem \eqref{eq: bilevel task 1 b} coincides with the original bilevel optimization problem \eqref{eq: bilevel task 1} under the assumption that \( g(x, \cdot) \) is convex and/or satisfies   {weak PL condition} \cite{csiba2017global} for any $x$. In scenarios where \( g(x, \cdot) \) lacks these conditions, the reformulation in \eqref{eq: bilevel task 1 b} corresponds to finding a minimizer of the upper-level function over stationary solutions of the lower-level problem.

\subsection{Related Works}\label{sec:related_works}

\begin{table*}[t!]\scriptsize
     \renewcommand{\arraystretch}{1.0}
    \centering
    \begin{tabular}{l l l l l l}
        \toprule
        \textbf{Algorithm} & $g(x,\cdot)$ & \textbf{Additional Assumptions} &   Oracle & Loop(s) & \textbf{Complexity} \\
        \midrule
        BOME \cite{liu2022bome} & PL  & Bounded $|f|$ and $|g|$, lower-level unique solution  & 1st & Double & $\mathcal{O}(\epsilon^{-3})$ \\
        V-PBGD \cite{shen2023penalty} & PL  & See \cref{rem:discussion}  & 1st  & Double & $\tilde{\mathcal{O}}(\epsilon^{-3/2})$\\
        $F^2$BA \cite{chen2024finding} & PL & $\mu$-PL penalty function, $f$ has Lipschitz Hessians in y   & 1st  & Double & $\tilde{\mathcal{O}}(\epsilon^{-1})$\\
        GALET \cite{xiao2023generalized} & {PL + SC}  & {$\inf_{x,y}\{\sigma^{+}_{\min}(\nabla_{yy}^2 g(x, y))\}  > 0 \quad \forall{x,y}$}  & 2nd  & Triple & $\tilde{\mathcal{O}}(\epsilon^{-1})$\\
        HJFBiO \cite{huang2024optimal} & PL  & {$\sigma(\nabla_{yy}^2 g(x, y^\star(x)) ) \in [\mu,L_g]$}  & {2nd } & Single & $\mathcal{O}(\epsilon^{-1})$\\
        \textbf{Ours-\cref{thm:conv_rate}} & non-convex & Regularity Condition, see \cref{rem:regularity_condition}  & 2nd  & Single  & $\mathcal{O}(\epsilon^{-3/2})$\\
        \textbf{Ours-\cref{thm:conv_rate2}} & non-convex & \qquad \qquad \qquad \qquad \qquad -  & 2nd  & Single & $\mathcal{O}(\epsilon^{-3})$\\
        \bottomrule
    \end{tabular}
    \caption{Comparison of bilevel algorithms with non-convex lower-level specifically without the presence of lower-level constraints. \textbf{PL} and \textbf{SC} stand for the Polyak-Łojasiewicz and strong convexity, respectively. $\sigma(A)$ represents the singular values of matrix A. Complexity is based on finding an $\epsilon$-stationary solution such that $\|\nabla \ell(x)\|^2 \leq \epsilon$ or its equivalent variants. The notation $\tilde{\mathcal{O}}$ omits the dependency on $\log(\varepsilon^{-1})$.  {For more discussion on related works see \cref{app:Related_works}.}}
    \label{tab:compare}
\end{table*}

 {In this section, we present a comprehensive review of recent works that are closely related to bilevel optimization problems with a non-convex lower level.}
% Several methods have been developed to address bilevel optimization problems relying on the optimality conditions of the lower-level problem to derive an equivalent single-level formulation \cite{hansen1992new}. However, large-scale or non-convex lower-level problems pose significant obstacles to this approach. Recently, more efficient gradient-based methods for bilevel optimization have emerged, broadly categorized as the approximate implicit differentiation (AID) approach \cite{pedregosa2016hyperparameter,domke2012generic,ghadimi2018approximation} and the iterative differentiation (ITD) approach \cite{shaban2019truncated,Franceschi_ICML18,grazzi2020iteration}.

While most works assume a unique lower-level solution, newer studies address cases where this assumption does not hold \cite{sow2022constrained,chen2023bilevel,shen2024method,liu2022bome,xiao2023generalized}.
Some works focus on bilevel optimization with a convex lower-level problem, which introduces the challenge of multiple lower-level solutions \cite{liu2020generic,sow2022constrained,liu2023averaged, shen2024method, shen2023penalty, chen2023bilevel,lu2024first}. 
However,  \cite{chen2023bilevel} has shown that additional assumptions on the lower-level problem are necessary to ensure meaningful guarantees. 

{Beyond lower-level convexity, recently, several research studied bilevel optimization with a non-convex lower-level problem satisfying PL condition. 
More specifically,  \cite{liu2022bome} proposed a first-order method and established the first non-asymptotic convergence guarantee of $\mathcal{O}(\epsilon^{-3})$ for bilevel optimization satisfying the PL condition. They further assume that the lower-level solution is unique, and both the upper and lower-level objective functions are bounded. 
Later, \cite{shen2023penalty} introduced a penalty-based algorithm in which the lower-level objective $g$ satisfies the PL condition, where the method relies solely on first-order oracles with an iteration complexity of $\mathcal{O}(\epsilon^{-3/2})$. 
 \cite{kwon2023penalty} studied the non-convex
bilevel optimization with non-convex lower-level satisfying
proximal error-bound (EB) condition that is analogous
to PL condition when the lower-level is unconstrained. Their approach guarantees convergence to an \( \epsilon \)-stationary point of the penalty function, requiring \( \tilde{\mathcal{O}}(\epsilon^{-3/2}) \) first-order gradient oracle calls. Further \cite{chen2024finding} showed that the proximal
operator in \cite{kwon2023penalty} is unnecessary and their method can converge under PL condition with an improved complexity of \( \tilde{\mathcal{O}}(\epsilon^{-1}) \). However,  They assume that the penalty function $h_\theta= \theta f + g$ satisfies the $\mu$--PL condition which considerably stronger and more restrictive than simply assuming that the lower-level function $g$ is $\mu$--PL. This is because it is not even straightforward to ensure that the sum of two PL functions remains PL. Moreover, assuming that the Hessian of the upper-level objective function is Lipschitz continuous is a particularly restrictive condition that further narrows the class of problems to which the analysis applies.  \cite{xiao2023generalized} proposed a generalized alternating method and obtained an $\epsilon$-stationary point within $\tilde{\mathcal{O}}(\epsilon^{-1})$ under PL condition of the lower-level problem. 
 \cite{huang2023momentum} introduced a class of momentum-based gradient method for non-convex bilevel optimization problems, where both upper-level
and lower-level problems are non-convex, and the lower-level problem satisfies PL condition. Furthermore, they assume that $\nabla^2_{yy} g(x,y^*(x))$ is non-singular at the minimizer of $g$. Their method achieves a complexity of $\tilde{\mathcal{O}}(\epsilon^{-1})$ in finding an $\epsilon$-stationary solution. However, their proposed method requires computing expensive projected Hessian and Jacobian matrices along with their inverses. Moreover, computing the SVD decomposition of the Hessian matrix at each step impose a $\mathcal{O}(d^3)$ complexity where $d = \max\{m,n\}$. More recently,  \cite{huang2024optimal} claimed that the projection operator can remove expensive SVD decomposition in \cite{huang2023momentum} and proposed a Hessian/Jacobian-free
bilevel method achieving a complexity of $\mathcal{O}(\epsilon^{-1})$ in finding $\epsilon$-stationary solution under the same setting.  {A concise
comparison between our proposed method and  related works is summarized
in \cref{tab:compare}}.
% Another approach to tackling non-convex bilevel optimization involves relaxation techniques, such as incorporating regularization in the lower-level problem \cite{liu2021value, mehra2021penalty} or approximating the lower-level optimal solution set with its $\epsilon$-optimal solutions \cite{lin2014solving,shen2023penalty}.
Some techniques additionally handle equality and inequality constraints at the lower level \cite{xiao2023alternating, khanduri2023linearly, xu2023efficient, kornowski2024first}.}
%Most relevant to our method, \cite{liu2022bome} employs dynamic barriers to solve \eqref{eq: bilevel task 1} using a first-order approach.

\subsection{Contribution} 
In this paper, we study a general bilevel optimization problem with smooth objective functions. While existing methods primarily focus on the nonconvex-PL setting, we consider a more general nonconvex-nonconvex bilevel framework. Drawing inspiration from the Relaxed Gradient Flow (RXGF) \cite{sharifi2025safe}, we introduce a novel discrete-time gradient-based method for efficiently solving bilevel optimization problems.   
RXGF is a continuous-time dynamical system that ensures convergence to a neighborhood of the optimal solution, assuming the lower-level problem is strongly convex. However, its reliance on solving ODEs makes it computationally impractical for large-scale problems. To address this limitation, we propose an efficient discretized variant of RXGF that is computationally feasible for high-dimensional settings. We establish that, under the existence of a KKT point and a regularity assumption, our method achieves a convergence rate of \( \mathcal{O}(1/\epsilon^{1.5}) \) in terms of the squared norm of the KKT residual. Moreover, even without the regularity assumption, we demonstrate that our method attains a convergence rate of \( \mathcal{O}(1/\epsilon^{3}) \) for the same metric.
\Cref{fig:Overview} presents an overview of the methods.

% 
%We demonstrate the effectiveness and scalability of our method through numerical experiments showcasing its ability to efficiently navigate the bilevel optimization landscape and find high-quality solutions.

\begin{figure}[t!]
    \centering
    \includegraphics[width=0.65\linewidth]{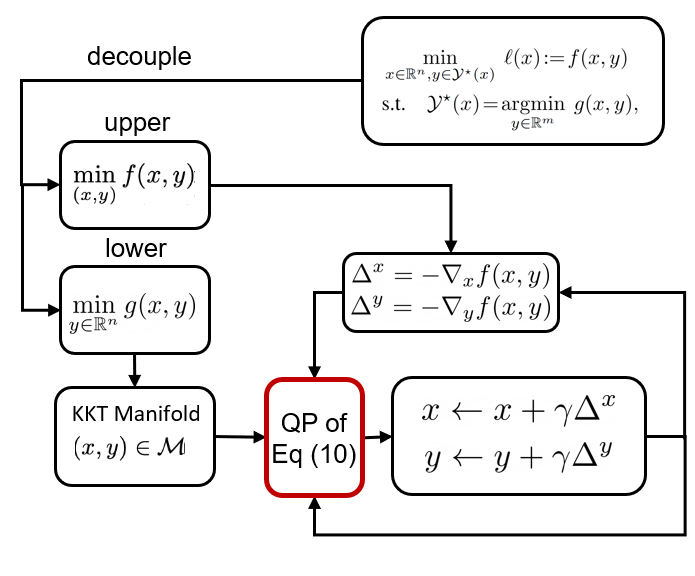}
    \caption{
        Overview of the proposed method. 
        The QP takes the gradient directions, perturbs them according to the lower-level problem, and then the variables are updated using the new directions. 
    }
    \label{fig:Overview}
\end{figure}

\section{Preliminaries} 
% \subsection{Notation}
% %An extended class $\mathcal{K}$ function is a function $\alpha:(-b, a) \mapsto \mathbb{R}$ for some $a, b > 0$ that is strictly increasing and satisfies $\alpha(0) = 0$. We denote $L_r^{-}(h) = \{x \in \mathbb{R}^n \mid h(x) \leq r \}$ as the $r$-sublevel set of the function $h(\cdot)$. 
% We denote the Euclidean norm by $\|\cdot\|$ and $[x]_+:= \max\{0,x\}$. 
% The function $h(\cdot)$ is called $L^h$-Lipschitz if for any $z, z'$ we have  $\|h(z) - h(z')\| \leq L^h \|z - z'\|$. 
% \todo[inline]{If the list of notations that we need is thin, we should remove it and give their definitions in the first place we use.}
%Given a map  $F(x): \mathbb{R}^n  \to \mathbb{R}^m$, its Jacobian matrix is denoted by \( D_xF(x)\in\mathbb{R}^{m\times n} \).
%Furthermore, we denote the Euclidean projection of $\bar{x} \in \mathbb{R}^n$ onto a closed convex set $\mathcal{D} \subseteq \mathbb{R}^n$ as $\mathrm{Proj}_\mathcal{D}(\bar{x}):=\argmin_{x\in\mathcal{D}}\|x-\bar{x}\|^2$.

\subsection{Assumptions and Definitions}
This subsection outlines the definitions and assumptions required throughout the paper. We begin by discussing the assumptions on the upper-level and lower-level objective functions,
respectively.

\begin{assumption}[Upper-level Objective]\label{assumption:upperlevel}
$f:\reals^n\times\reals^m\to\reals$ is a continuously differentiable function such that $\bar{f}\triangleq \inf_{(x,y)} f(x,y)>-\infty$.
$\nabla_x f(\cdot,\cdot)$ and $\nabla_y f(\cdot,\cdot)$ are $L_{x}^f$ and $L^f_{y}$-Lipschitz continuous, respectively. Moreover, there exists $C_f>0$ such that $\|\nabla f(x,y)\|\leq C_f$ for any $(x,y)$. 

\end{assumption}
\begin{assumption}[Lower-level Objective]\label{assumption:lowerlevel}
 For any $x$, function $g(x,\cdot)$ is twice continuously differentiable. $\nabla_y g(x,\cdot)$, and $\nabla_y g(\cdot,y)$, are $L_{yy}^g$- and $L^g_{yx}$-Lipschitz continuous for all $(x,y)$, respectively. There exists $C_g>0$ such that $\|\nabla_y g(x,y)\|\leq C_g$ for all $(x,y)$.
\end{assumption}
%\mahyar{We need transpose for second derivatives of g as well as power 2 for nabla for consistent notation.}
\begin{assumption}\label{assum:KKT}
    There exists $(\bar{x},\bar{y})$ and $\bar{\lambda}\in\mathbb{R}^m$ such that $\nabla_y g(\bar{x},\bar{y})=0$, $\nabla_x f(\bar{x},\bar{y})+\nabla_{yx}^2 g(\bar{x},\bar{y})^\top\bar{\lambda}=0$, and $\nabla_y f(\bar{x},\bar{y})+\nabla_{yy}^2 g(\bar{x},\bar{y})^\top\bar{\lambda}=0$.
\end{assumption}

Consider the following single-level reduction of the bilevel problem in \eqref{eq: bilevel task 1}
\begin{alignat}{2}\label{eq: bilevel task 2 a}
    & \min_{(x,y)} \ f(x,y) \quad 
    &\text{s.t.} \quad (x,y) \in \mathcal{M}
\end{alignat}
where $\mathcal{M} = \{(x,y) \mid  h(x,y):=\|\nabla_y g(x,y)\|_2^2=0\}$. 
This formulation has only one constraint compared to \eqref{eq: bilevel task 1 b}. Our goal is to find an $\epsilon$-approximate KKT point of the above problem which we define as follows.
\begin{definition}\label{def:epsilon-KKT}
$(x_\epsilon,y_\epsilon,\lambda_\epsilon)\in\mathbb{R}^n\times \mathbb{R}^m\times \mathbb{R}$ is called an $\epsilon$-KKT point if $h(x_\epsilon,y_\epsilon)\leq \epsilon$ and $\|\nabla f(x_\epsilon,y_\epsilon)+\lambda_\epsilon \nabla h(x_\epsilon,y_\epsilon)\|^2\leq \epsilon$. 
\end{definition}
Note that if $(x_\epsilon,y_\epsilon,\lambda_\epsilon)$ is an $\epsilon$-KKT of problem \eqref{eq: bilevel task 2 a}, then $(x_\epsilon,y_\epsilon,\nu_\epsilon)$ with $\nu_\epsilon\triangleq \lambda_\epsilon\nabla_y g(x_\epsilon,y_\epsilon)\in\mathbb{R}^m$ is an $\epsilon$-KKT of problem \eqref{eq: bilevel task 1 b}, i.e.,
\begin{align*}
        &\|\nabla_y g(x_\epsilon,y_\epsilon)\|^2\leq \epsilon,\\
        &\|\nabla_x f(x_\epsilon,y_\epsilon)+\nabla_{yx}^2 g(x_\epsilon,y_\epsilon)^\top\nu_\epsilon\|^2\leq \epsilon,\\
        &\|\nabla_y f(x_\epsilon,y_\epsilon)+\nabla_{yy}^2 g(x_\epsilon,y_\epsilon)^\top\nu_\epsilon\|^2\leq \epsilon.
    \end{align*}
The sufficient conditions for the existence of a KKT solution for problem \eqref{eq: bilevel task 1 b} have been studied in several studies \cite{ye1995optimality,xiao2023generalized}. For example, in \cite{liu2022bome}, a constant rank constraint qualification (CRCQ) is assumed to guarantee the existence of KKT points. Meanwhile,  \cite{xiao2023generalized} explored the calmness condition and demonstrated that the PL condition for \( g(x, \cdot) \) ensures the existence of a KKT point. However, since our focus is on finding an \(\epsilon\)-KKT point as defined above, we adopt a more general assumption and only require that such a point exists.%

% \todo[inline]{We should remove the following text up to section 3 and briefly explain the continuous-time dynamic as an existing and well-established approach that motivated us to develop our algorithm.}

\subsection{ {Safe} Gradient Flow for Bilevel Optimization}
In the context of bilevel optimization,  \cite{sharifi2025safe} recently introduced a safe gradient flow mechanism to solve \eqref{eq: bilevel task 2 a}, where the associated ODE is obtained by solving the following convex quadratic program (QP): 
\begin{align}\label{eq:convex QP problem 1}
    (\dot{x},\dot{y}):= \argmin_{(\dot{x}_d, \dot{y}_d)} & \ \frac{1}{2} \|\dot{x}_d + \nabla_x f\|_2^2 + \frac{1}{2} \|\dot{y}_d + \nabla_y f\|_2^2 \\
    \text{s.t.} & \ \nabla_x h^\top \dot{x}_d + \nabla_y h^\top \dot{y}_d + \alpha h = 0. \nonumber
\end{align}
This QP minimally perturbs the gradient flow on the upper-level objective to enforce exponential convergence of $h$ to zero. The closed-form solution of the QP leads to the ODE,
\begin{equation} \label{eq: gradient flow upper level safe 2} 
\begin{aligned}
    \dot{x} &= -\nabla_x f - \lambda \nabla_{x}h, \\
    \dot{y} &= -\nabla_y f - \lambda\nabla_y h, \\
    \lambda\!&= \begin{cases}
         \frac{-\nabla_x h^\top \nabla_x f\! -\! \nabla_{y}h^\top \nabla_y f\! +\! \alpha h}{\|\nabla_{x}h\|^2 \! + \! \|\nabla_{y}h\|^2}\ & h \neq 0 \\
         0 & h=0
    \end{cases}
\end{aligned}
\end{equation}
Owing to the constraint in \eqref{eq:convex QP problem 1}, this ODE ensures $\dot{h}+\alpha h =0$ along the trajectories, implying exponential convergence of $h$ to zero. As another appealing feature, \eqref{eq: gradient flow upper level safe 2} does not involve any matrix inversion calculation. It is proven in \cite{sharifi2025safe} that when $g$ is strongly convex, a variant of this ODE converges to an $\epsilon$-neighborhood of a stationary solution of \eqref{eq: bilevel task 1} at an $\mathcal{O}(1/t)$ rate ($\epsilon$ arbitrary).

One might be tempted to directly discretize \eqref{eq: gradient flow upper level safe 2} using, for example, the Euler method to derive an iterative algorithm. However, ensuring exponential convergence to the manifold \( \mathcal{M} \) imposes strong assumptions on the lower-level objective \( g \). Specifically, the full rank condition of the matrix \( [\nabla_{yx}^2 g \ \nabla_{yy}^2 g] \) guarantees that the QP \eqref{eq:convex QP problem 1} is always feasible, which in turn ensures the exponential convergence of \( h \) to zero.

In the following section, we will design a quadratic program directly in discrete time that guarantees convergence to a stationary point without requiring strong assumptions on the lower-level objective \(g\).

%a discretization of the proposed ODE that ensures any-time
%feasibility while guaranteeing finding a stationary point.

\section{Proposed Method}
In this section, our goal is to design an iterative algorithm of the form
\begin{align}
    x_{k+1} &= x_k + \gamma \Delta_k^x \notag\\
    y_{k+1} &= y_k + \gamma \Delta_k^y \notag
\end{align}
where $\gamma>0$ is the stepsize and $\Delta_k^x, \Delta_k^y$ are search directions. Inspired by \eqref{eq:convex QP problem 1}, we propose the following QP to obtain these directions, 
\begin{align}\label{eq: first-order lower-level discrete time2}
(\Delta_k^x,\Delta_k^y)=&\argmin_{\Delta^x,\Delta^y}~\frac{1}{2}\|\Delta^x+\nabla_x f_k\|^2_2+\|\Delta^y+\nabla_y f_k\|^2_2\\
&\hbox{s.t.}~\nabla_x h_k^\top \Delta^x + \nabla_y h_k^\top \Delta^y + \alpha \rho_k \leq 0, \nonumber
\end{align}
where $\nabla_x f_k = \nabla_x f(x_k,y_k)$, $\nabla_y f_k = \nabla_y f(x_k,y_k)$, $\nabla_x h_k = \nabla_x h(x_k,y_k)$, $\nabla_y h_k = \nabla_y h(x_k,y_k)$, and $\rho_k = \rho(x_k,y_k)$ is positive whenever $h_k=h(x_k,y_k)>0$ and is zero otherwise, ensuring reduction of infeasibility. The primal and dual solution to this QP is 
\begin{subequations}
\begin{align} 
 \Delta_k^x &=-\nabla_x f_k-\lambda_k\nabla_x h_k \label{eq:deltax_1}\\
 \Delta_k^y &=-\nabla_y f_k -\lambda_k \nabla_y h_k  \label{deltay:lam_1}\\
    \lambda_k &= 
\frac{\left[-\nabla_x h_k^\top \nabla_x f_k - \nabla_y h_k^\top \nabla_y f_k + \alpha \rho_k\right]_{+}}
{\|\nabla_x h_k\|^2 + \|\nabla_y h_k\|^2}. \label{eq:lam_1} 
\end{align}
\end{subequations}
where we use the notation $[x]_+\triangleq \max\{0,x\}$. Our proposed method is outlined in \cref{alg}.
% \ssh{Do we need a remark here to show that at the equilibrium, this recovers the KKT of the original BLO?}
% \ssh{
% \begin{remark}
%     The QP defined in \eqref{eq: first-order lower-level discrete time2}, at the equilibrium point $(\Delta_k^x,\Delta_k^y) = (0,0)$ recovers the KKT of the original problem \eqref{eq: bilevel task 1}???
% \end{remark}}
% \begin{proof}
%     At $(\Delta_k^x,\Delta_k^y) = (0,0)$, the KKT conditions of \eqref{eq: first-order lower-level discrete time2} can be written as
%     \begin{align}
        
%     \end{align}
% \end{proof}

\textbf{Choice of $\rho_k$:} The key distinction between \eqref{eq: first-order lower-level discrete time2} and \eqref{eq:convex QP problem 1} lies in the choice of \( \rho_k \), which plays a crucial role in ensuring convergence. To balance objective function value reduction with infeasibility reduction in problem \eqref{eq: bilevel task 2 a}, we consider two choices for the function \( \rho \): \textbf{(i)} \( \rho(x,y) = \|\nabla h(x,y)\|^2 \); \textbf{(ii)}  {\( \rho(x,y) = \|\nabla h(x,y)\| (h(x_0,y_0))^{1/2} \)}. Our first choice of function $\rho$ guarantees the reduction of the first-order KKT condition, leading to a stationary point of the constraint, i.e., $\|\nabla h(x,y)\|\leq \epsilon$. Combined with a commonly used regularity condition in nonconvex constrained optimization, this can be translated into an infeasibility result. Our second choice removes the requirement of such a regularity assumption at the cost of a slower convergence rate result. Both choices ensure that the QP \eqref{eq: first-order lower-level discrete time2} is always feasible. This follows from the fact that when $\nabla_x h =0$ and $\nabla_y h = 0$, we also have $\rho=0$.

\Cref{alg} outlines the proposed method, detailing the iterative procedure for solving the bilevel optimization problem \eqref{eq: bilevel task 1}.

    \begin{algorithm}[t!]
	\caption{Bilevel Approximation via Perturbed Gradient Descent}\label{alg}
	\begin{algorithmic}[1]
	\STATE \textbf{Input}: $\gamma,\alpha,C_0 >0$, $\rho:\reals^n\times\reals^m\to \reals$
	\STATE \textbf{Initialization}: $x_0 \in \mathbb R^n$, and $y_0 \in \mathbb R^m$ such that $\| \nabla_y g(x_0,y_0)\|^2 \leq \alpha^2 C_0$.
        % \STATE Choose $\gamma$ based on K
	\FOR{$k \geq 0$}
            \STATE Compute $\lambda_k$ using \eqref{eq:lam_1}
            \STATE $\Delta_k^x \gets -\nabla_x f(x_k,y_k)- \lambda_k \nabla_x h(x_k,y_k)$
            \STATE $\Delta_k^y \gets -\nabla_y f(x_k,y_k)- \lambda_k \nabla_y h(x_k,y_k)$
        
            \STATE $x_{k+1}\gets x_k+\gamma \Delta_k^x$
            \STATE $y_{k+1}\gets y_k+\gamma \Delta_k^y$
	\ENDFOR
	\end{algorithmic}
\end{algorithm}

% \begin{lemma}\label{lem:bound_lambda}
%     \red{Show that $\lambda_k$ is bounded}
% \end{lemma}

\begin{remark}[Computing $\nabla_x h$ and $\nabla_y h$]
The proposed methods require computing the gradients of \( h \), given by  
\begin{align}
    \nabla_x h &= \nabla_{yx}^2 g^\top  \nabla_y g, \notag \\ 
    \nabla_y h &= \nabla_{yy}^2 g^\top  \nabla_y g. \notag
\end{align}  
At first glance, this might suggest the need to store the Jacobian of \( g \), specifically \( \nabla_{yy}^2 g \) and \( \nabla_{yx}^2 g \). However, modern automatic differentiation frameworks such as PyTorch \cite{paszke2019pytorch} enable direct computation of the required matrix-vector products, eliminating the need to explicitly store these second-order derivatives. This significantly reduces the computational burden.  More specifically, using PyTorch’s \texttt{torch.autograd.grad} with \texttt{grad\_outputs=dgdy}, we can efficiently compute the necessary matrix-vector products without constructing the full Jacobian.
    % \[\footnotesize
    % \texttt{dgdyy @ dgdy = torch.autograd.grad(dgdy, y, grad\_outputs=dgdy)}
    % \]
    % \[\footnotesize
    % \texttt{torch.autograd.grad(dhdy, y)}
    % \]
    %}
\end{remark}

\section{Convergence Analysis}
In this section, we establish the convergence properties of our proposed algorithm for solving \eqref{eq: bilevel task 2 a}. Our objective is to find a pair \((\bar{x}, \bar{y})\) that satisfies the \(\epsilon\)-KKT conditions defined in Definition \ref{def:epsilon-KKT}. First, we prove a convergence rate of \(\mathcal{O}(1/K^{2/3})\) under a regularity assumption when choosing \(\rho(x,y) = \|\nabla h(x,y)\|^2\). Furthermore, we show that by setting \(\rho(x,y) = \|\nabla h(x,y)\|(h(x_0,y_0))^{1/2}\), we achieve a convergence rate of \(\mathcal{O}(1/K^{1/3})\) without requiring any regularity assumption.
\begin{theorem}\label{thm:conv_rate}
      Suppose that Assumptions \ref{assumption:upperlevel} and \ref{assumption:lowerlevel} hold and $\rho=\|\nabla h\|^2$.
     Let $\{(x_k,y_k,\lambda_k)\}_{k=0}^{K-1}$ be the sequence generated by \cref{alg} with $C_0>0$ and step size $\gamma >0$  such that $\gamma\leq \min\{\alpha,\frac{1}{L_f+\alpha L_h}\}$. Then for all $K \geq 1$, 
     \begin{align*}
    &\frac{1}{K}\sum_{k=0}^{K-1}\left(\|\Delta_k^x \|^2 \!+\!\|\Delta_k^y \|^2\right)\leq \frac{4(f_0+\alpha^3 C_0-\bar{f})}{\gamma K}+\frac{2\alpha C_0}{\gamma L_h K} +  2{\alpha^2  L_h}C_f^2,
    % +\frac{2\delta^2}{L_h}
     \end{align*}
and, 
     \begin{align*}
    &\frac{1}{K}\sum_{k=0}^{K-1}\left(\|\nabla_x h_k\|^2  +\|\nabla_y h_k\|^2\right) \leq \frac{2\alpha C_0}{\gamma K}+\frac{2L_h(f_0+\alpha^3 C_0-\bar{f})}{\gamma K} +  {\alpha^2  L_h^2}C_f^2
\end{align*}
\end{theorem}
\begin{proof}
    See \Cref{proof:thm_conv_rate}.
\end{proof}

\begin{corollary}\label{cor:rate-nabla-h}
    Let $\{(x_k,y_k,\lambda_k)\}_{k=0}^{K-1}$ be the sequence generated by \cref{alg} with $\alpha=K^{-1/3}$ and $\gamma=\min\{\alpha,\frac{1}{L_f+\alpha L_h}\}$. Under the premises of \cref{thm:conv_rate} we have that $\frac{1}{K} \sum_{k=0}^{K-1} \| \Delta_k^x\|^2 + \| \Delta_k^y\|^2 \leq \mathcal{O}(1/K^{2/3})$ and $\frac{1}{K} \sum_{k=0}^{K-1} \| \nabla h_k\|^2 \leq \mathcal{O}(1/K^{2/3})$. Therefore, there exists $t\in\{0,\hdots, K-1\}$ such that
    \begin{align*}
        \max\{\|\nabla h_t\|^2,\|\nabla f_t+\lambda_t\nabla h_t\|^2\}\leq \epsilon,
    \end{align*}
    within $K=\mathcal{O}(1/\epsilon^{1.5})$ iterations.
\end{corollary}
\begin{proof}
    See \cref{proof:corollary_rate_nabla_h}.
\end{proof}
% \ssh{How is the regularity condition defined in this corollary different from simply assuming PL condition for $h$? and if it is the same, then how does it compare with the other papers that assume PL for $g$? Is it more or less restrictive?}
\begin{corollary}\label{cor:rate_nabla_g}
    Suppose the following regularity assumption hold: there exists $c>0$ such that $\|\nabla_y g(x,y)\|\leq c\|\nabla h(x,y)\|$ for any $(x,y)$. Under the premises of Corollary \ref{cor:rate-nabla-h} there exists, $t\in\{0,\hdots, K-1\}$ such that $(x_t,y_t,\lambda_t)$ is an $\epsilon$-KKT point of problem \eqref{eq: bilevel task 2 a}, i.e, 
    \begin{align*}
        \max\{\|\nabla_y g_t\|^2,\|\nabla f_t+\lambda_t\nabla h_t\|^2\}\leq \epsilon,
    \end{align*}
    within $K=\mathcal{O}(1/\epsilon^{1.5})$ iterations.
\end{corollary}
  % \todo[inline]{\ssh{This corollary is more strong that stating that 
  %  $\nabla_y g \notin \mathcal{N}([\nabla_{yy} g, \nabla_{yx} g])$.\\
  %  Assume that there exists a $\nabla_y g \neq 0$ that is in the null space of that matrix, then $\nabla h = 0$, so no $c$ exists.
  %  }}
\begin{proof}
    See \cref{proof:corollary_rate_nabla_g}.
\end{proof}
\begin{remark}\label{rem:regularity_condition}
 {The regularity condition used in Corollary \ref{cor:rate_nabla_g} is equivalent to the function $h$
satisfying the PL condition, which can be a more restrictive requirement than merely assuming that the lower-level function $g$ is PL}. However, this is a commonly used regularity condition in the optimization literature and has been employed in the convergence analysis of iterative algorithms for non-convex optimization problems with functional constraints \cite{bolte2018nonconvex,sahin2019inexact,li2021rate,lin2022complexity, lu2022single, li2024stochastic} corresponding to problem \eqref{eq: bilevel task 1 b}. The properties and practical relevance of this assumption have been extensively examined in \cite{bolte2018nonconvex,sahin2019inexact}, where it is shown to be a weaker condition than the Mangasarian-Fromovitz Constraint Qualification (MFCQ). In the context of bilevel optimization, a notable example of a lower-level function $g$ that satisfies this regularity assumption is \( g(x,y) = \frac{1}{2} \|Ay - Bx\|^2 \), where \( A \in \mathbb{R}^{p \times m} \) and \( B \in \mathbb{R}^{p \times n} \) which arises as a loss function in various applications, including robust and adversarial learning.
%\naz{The intuition of using regularity assumption is to prevent scenarios  where the gradient of $g$ could dominate the updates, potentially leading to infeasibility in the iterative process. By bounding $\|\nabla_y g\|$ with a multiple of $\| \nabla h\|$, the algorithm ensures that each update step $(\Delta_k^x, \Delta_k^y)$ effectively reduces both the objective and the constraint violation in a balanced manner, which is essential for achieving convergence to an $\epsilon$-KKT point, as it guarantees that progress toward optimizing $f(x,y)$ does not compromise the feasibility dictated by $h(x,y)$. Regularity conditions play a crucial role in solving constrained optimization problems, including Slater’s condition and the Linear Independence Constraint Qualification (LICQ), which distinguish these problems from unconstrained settings or those with closed-form projection operators. The regularity assumption utilized in our work is a widely used assumption and has been employed in the convergence analysis of iterative algorithms for non-convex optimization problems with functional constraints \cite{sahin2019inexact,li2021rate,lin2022complexity, lu2022single, li2024stochastic}. }
\end{remark}

While the considered regularity assumption in the above result leads to a favorable convergence rate for finding an $\epsilon$-KKT point, it imposes limitations on the class of optimization problems to which our proposed algorithm can effectively be applied. In response to this limitation, we next show that by selecting a different function $\rho$, we can eliminate the need for the stringent regularity assumption.

% \begin{alignat}{2}\label{eq: first-order lower-level discrete time2}\tag{DRXSF2}
%     (\Delta_k^x,\Delta_k^x) = &\argmin_{(\Delta^x,\Delta^y)}~ &&\frac{1}{2}\|\Delta^x+\nabla_x f_k)\|_2^2  + \frac{1}{2}\|\Delta^x+\nabla_y f_k\|_2^2\\
%     &\text{s.t.}\quad && \nabla_x h_k^\top \Delta^x \!+\! \nabla_y h_k ^\top \Delta^y \nonumber \\
%     & && \!+\! \alpha (\|\nabla h_k\|\sqrt{h_k}) \leq 0.\notag    
%     \end{alignat}
% % where $z_k \;=\; ( x_k , y_k ) \in \mathbb{R}^{n+m}$, and $\Delta_k \;=\; (\Delta^x_k , \Delta^y_k ) \in \mathbb{R}^{n+m}$, 
% then dual multiplier is:
% \begin{equation}\label{eq:lam_2}
%     \lambda_k = 
% \resizebox{0.94\columnwidth}{!}{$
% \frac{\left[-\nabla_x h_k^\top \nabla_x f_k - \nabla_y h_k^\top \nabla_y f_k + \alpha \left(\|\nabla h_k\| \sqrt{h_k}\right)\right]_{+}}
% {\|\nabla_x h_k\|^2 + \|\nabla_y h_k\|^2}
% $}
% \end{equation}

\begin{theorem}\label{thm:conv_rate2}
      Suppose that  Assumptions \ref{assumption:upperlevel} and \ref{assumption:lowerlevel} hold and  {$\rho(x,y)=\|\nabla h(x,y)\|(h(x_0,y_0))^{1/2}$}.
     Let $\{(x_k,y_k,\lambda_k)\}_{k=0}^{K-1}$ be the sequence generated by \cref{alg} with $C_0>0$ and stepsize $\gamma >0$  such that $\gamma=\min\{\frac{1}{K^{2/3}},\frac{1}{L_f}\}$. Define   {$B_\Delta\triangleq 2C_f+\alpha^2 \sqrt{C_0}$}, then for all $K \geq 1$, 
     \begin{align*}
    &\frac{1}{K}\sum_{k=0}^{K-1}\left(\|\Delta_k^x \|^2 +\|\Delta_k^y \|^2\right)
    %&\leq \frac{2(f_0-f_K)}{\gamma K}+\alpha^2 B_\Delta^2 + \alpha C_0 +\frac{\gamma^2L_h B_\Delta^2}{2}\frac{(K-1)}{2}.\\
    \leq \frac{2(f_0-\bar{f})}{\gamma K}+\alpha^2 (B_\Delta^2 +C_0)%+\frac{\gamma^2L_h B_\Delta^2}{2}\frac{(K-1)}{2},
    \end{align*}
and, 
     \begin{align*}
    \frac{1}{K}\sum_{k=0}^{K-1} h_k\leq \alpha^2 C_0+\frac{\gamma^2L_h B_\Delta^2}{2}\frac{(K-1)}{2}.
    \end{align*}
\end{theorem}
\begin{proof}
    See \Cref{proof:thm_conv_rate_2}.
\end{proof}

\begin{corollary}\label{cor:rate_nabla_g2}
     Let $\{(x_k,y_k,\lambda_k)\}_{k=0}^{K-1}$ be the sequence generated by \cref{alg} with $\alpha=K^{-1/6}$ and $\gamma=\min\{\frac{1}{K^{2/3}},\frac{1}{L_f}\}$. Under the premises of \cref{thm:conv_rate2} we have that $\frac{1}{K} \sum_{k=0}^{K-1} \| \Delta_k^x\|^2 + \| \Delta_k^y\|^2 \leq \mathcal{O}(1/K^{1/3})$ and $\frac{1}{K} \sum_{k=0}^{K-1}   h_k \leq \mathcal{O}(1/K^{1/3})$. Therefore, there exists $t\in\{0,\hdots, K-1\}$ such that $(x_t,y_t,\lambda_t)$ is an $\epsilon$-KKT point of problem \eqref{eq: bilevel task 2 a}, i.e., 
    \begin{align*}
        \max\{\|\nabla_y g_t\|^2,\|\nabla f_t+\lambda_t\nabla h_t\|^2\}\!\leq \!\epsilon,
    \end{align*}
    within $K=\mathcal{O}(1/\epsilon^{3})$ iterations.
\end{corollary}
            
\begin{proof}
    See \cref{proof:corollary_rate_nabla_g2}.
\end{proof}
% \section{Adaptive Safe Step Size via Line Search}
% %
% In the previous section, we developed the iterative version of \eqref{eq: gradient flow upper level safe 2} using Euler discretization with constant step size and we proved asymptotic feasibility of the constraint set $L_{\varepsilon}^{-}(h)$ of the relaxed problem \eqref{eq:bilevel-approx}.  In this section, we propose a line search mechanism to ensure anytime feasibility of $L_{\varepsilon}^{-}(h)$. 

    \begin{figure*}[t]
    \centering
    \includegraphics[width=0.24\linewidth]{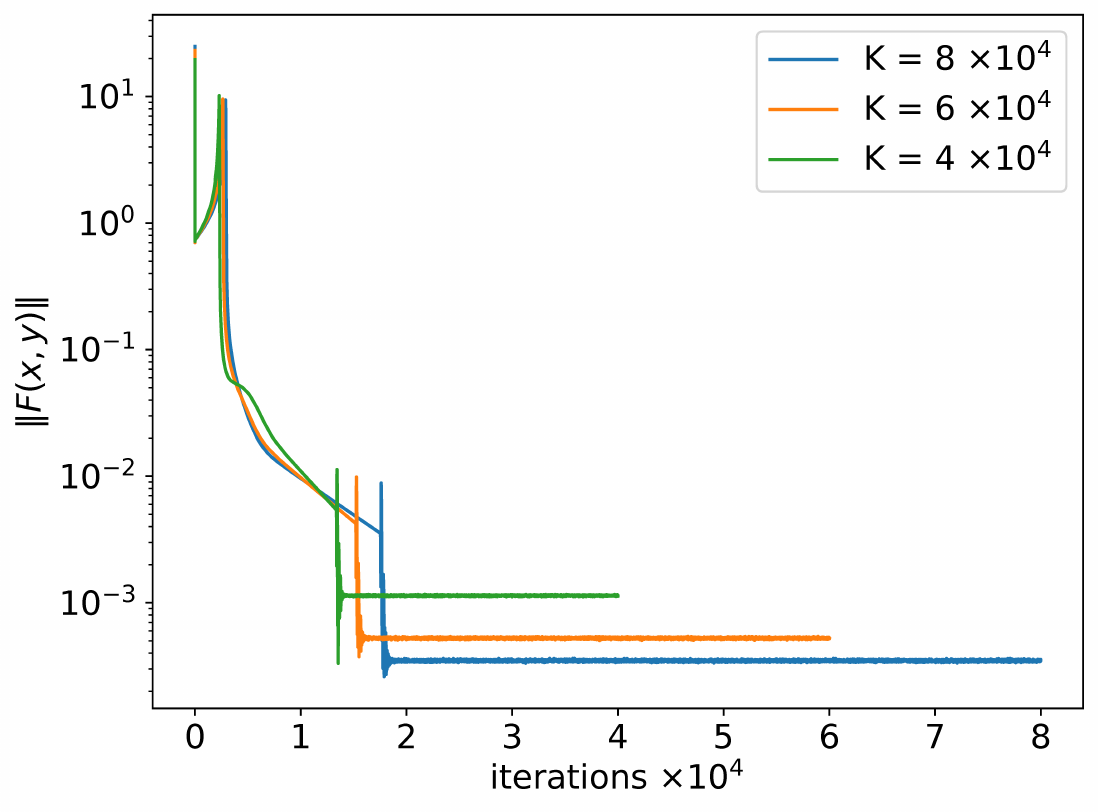}
    \includegraphics[width=0.24\linewidth]{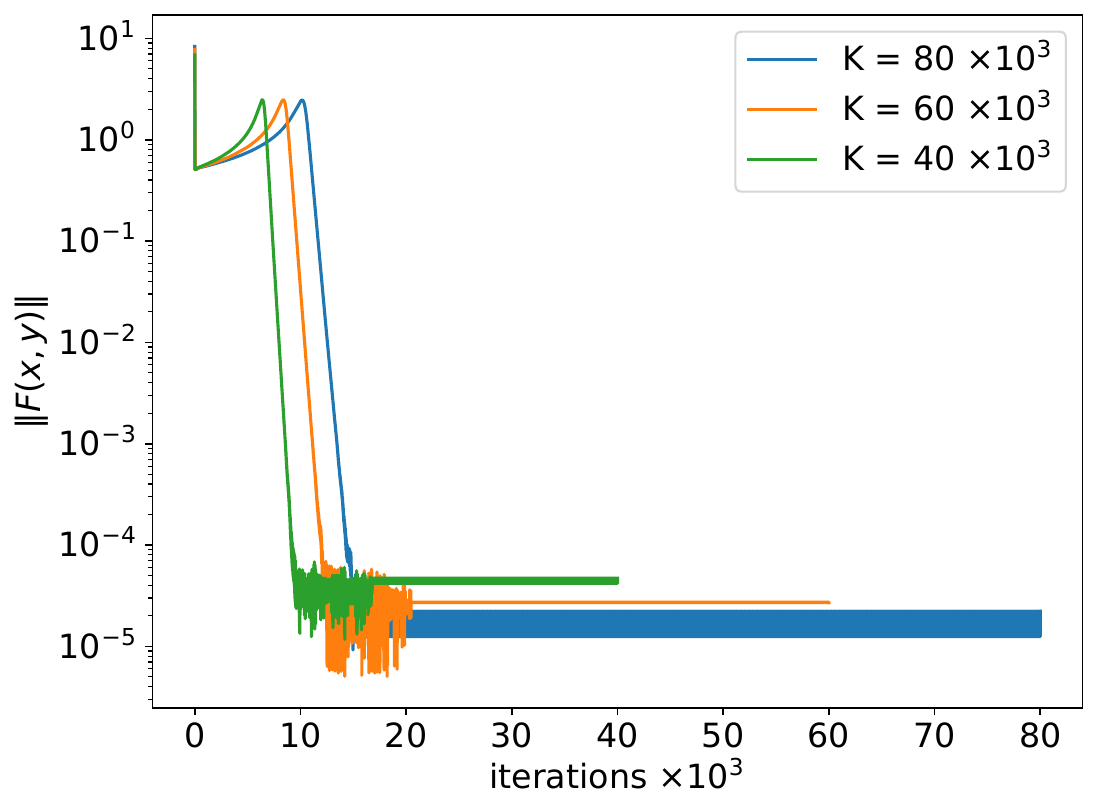}
    \includegraphics[width=0.24\linewidth]{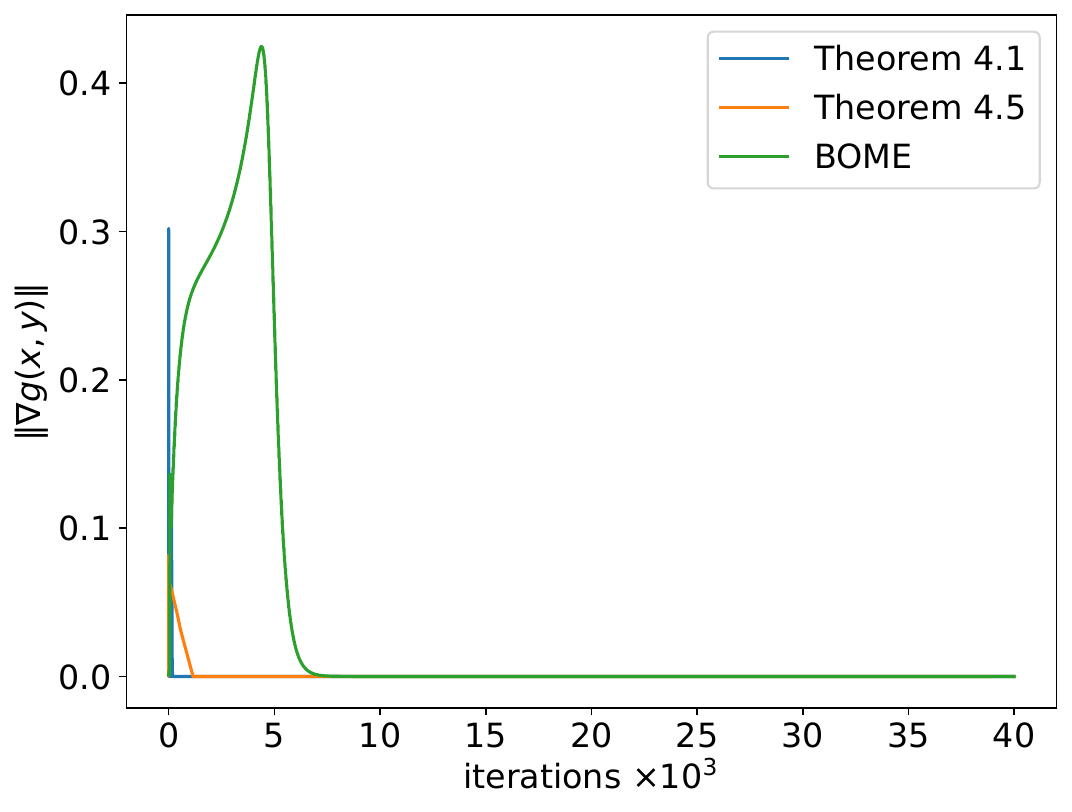}
    \includegraphics[width=0.24\linewidth]{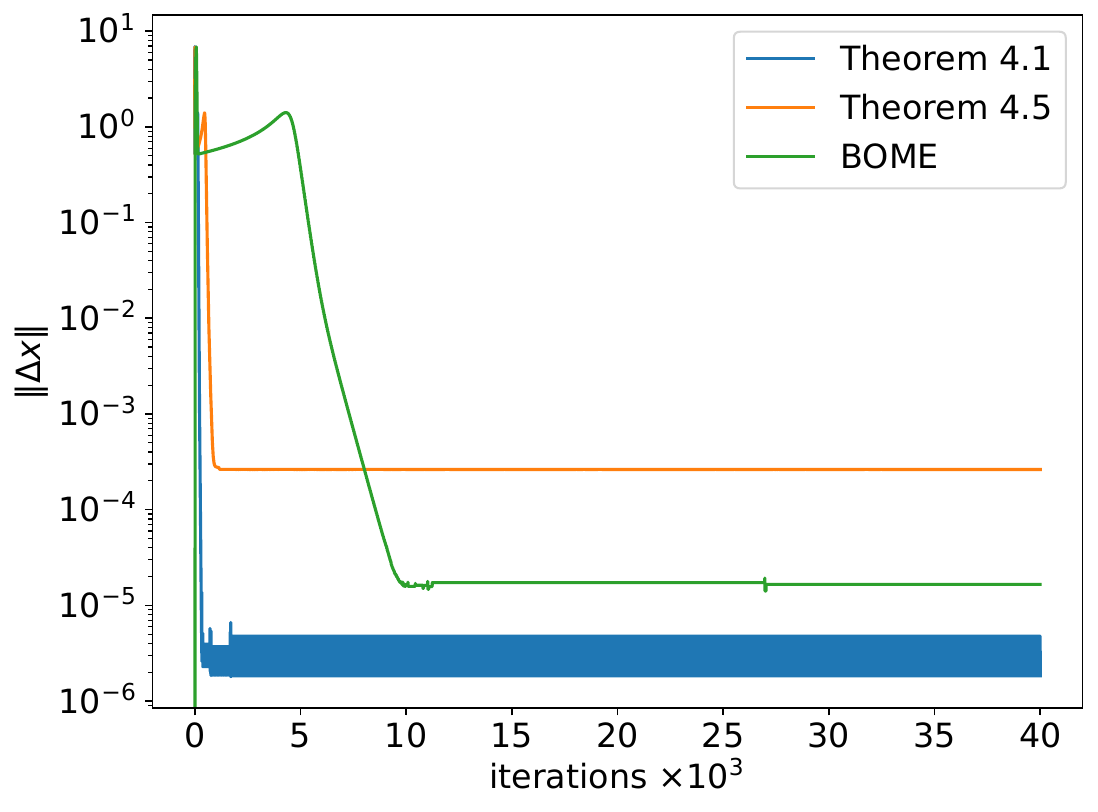}
    \caption{
    Effect of the number of iterations on the convergence on the strongly convex synthetic example problem and comparison with BOME \cite{liu2022bome} on the non-convex synthetic example benchmark.
    Parameter choice from \textit{(leftmost)} \Cref{thm:conv_rate} and \textit{(mid-left)} \Cref{thm:conv_rate2} on the synthetic example with the strongly convex lower-level function.
    \textit{(mid-right \& rightmost)} comparison with BOME \cite{liu2022bome} on the synthetic example with non-convex lower-level function.
    }
    \label{fig:experiment on K}
    \end{figure*}

\section{Numerical Experiments}
In this section, we numerically evaluate the performance of our method compared with other methods on both synthetic problems and a data hyper-cleaning (DHC) problem on the MNIST \cite{lecun2010mnist} dataset. 
For our methods, we include the result for different choices for $\alpha$ and $\gamma$ from Theorem (\ref{thm:conv_rate}) and Theorem (\ref{thm:conv_rate2}). For the experiments where the lower level function is strongly convex, we compare the norm of hyper-gradient ($\|F(x,y)\|$) as the metric for convergence. However, for the experiments where the lower level function is not strongly convex, we compare $\|\Delta\|$ since the hyper-gradient might not be well-defined.
 {We also provide additional experiments in \Cref{appendix:exp}.
We will publicly release our code after the review process.}

\subsection{Synthetic Example}
    To showcase the performance of our method, we start with two simple numerical examples.
    \subsubsection{Strongly Convex Lower-level} \label{subsection:toy_example_convex}
    Consider the following basic bilevel optimization problem 
    \begin{align}
            &\min_x \quad  \sin(c^\top x + d^\top y^\star(x)) +  \log(\|x+y^\star(x)\|^2 + 1) \notag \\
            &\text{s.t.} \quad  y^\star(x) \in \argmin_y  \tfrac{1}{2} \|Hy - x\|^2,\notag
    \end{align}
    where $x, y, c, d \in \mathbb{R}^{20}$ and $H \in \mathbb{R}^{20 \times 20}$ is randomly generated in a way such that its condition number is no larger than 10.

    The first two plots in \Cref{fig:experiment on K} show the reduction in the norm of hyper-gradient with respect to the number of iterations using parameter choices from both \Cref{thm:conv_rate} and \Cref{thm:conv_rate2}. In both cases, as we expect, we can see the algorithm converges to a more accurate solution as we increase the number of iterations, even though it takes longer to converge.
    
    % We can see that though the parameter choice from \Cref{thm:conv_rate2} leads to slower convergence in the beginning than the parameter choice from \Cref{thm:conv_rate}, it eventually finds a more accurate solution.

    \subsubsection{Non-convex Lower-level} \label{subsection:toy_example_nonconvex}
    To also test our method on a benchmark with a non-convex lower-level problem, we slightly change the previous setup and design a lower-level problem such that the it is non-convex. Consider the following bilevel optimization problem
    \begin{align}
            &\min_x \quad  \sin(c^\top x + d^\top y^\star(x)) +  \log(\|x+y^\star(x)\|^2 + 1) \notag \\
            &\text{s.t.} \quad  y^\star(x) \in \argmin_y  \cos(\tfrac{1}{2} \|Hy - x\|^2),\notag
    \end{align}
     where the parameters are generated in the same manner as in the previous subsection. 
     
     The last two plots in \Cref{fig:experiment on K} show the comparison between our method and BOME \cite{liu2022bome}, showcasing that our method remains close to the lower-level optimal point $y^\star(x)$, while reducing the KKT stationary condition.

     \begin{figure*}[t]
        \includegraphics[width=0.24\linewidth]{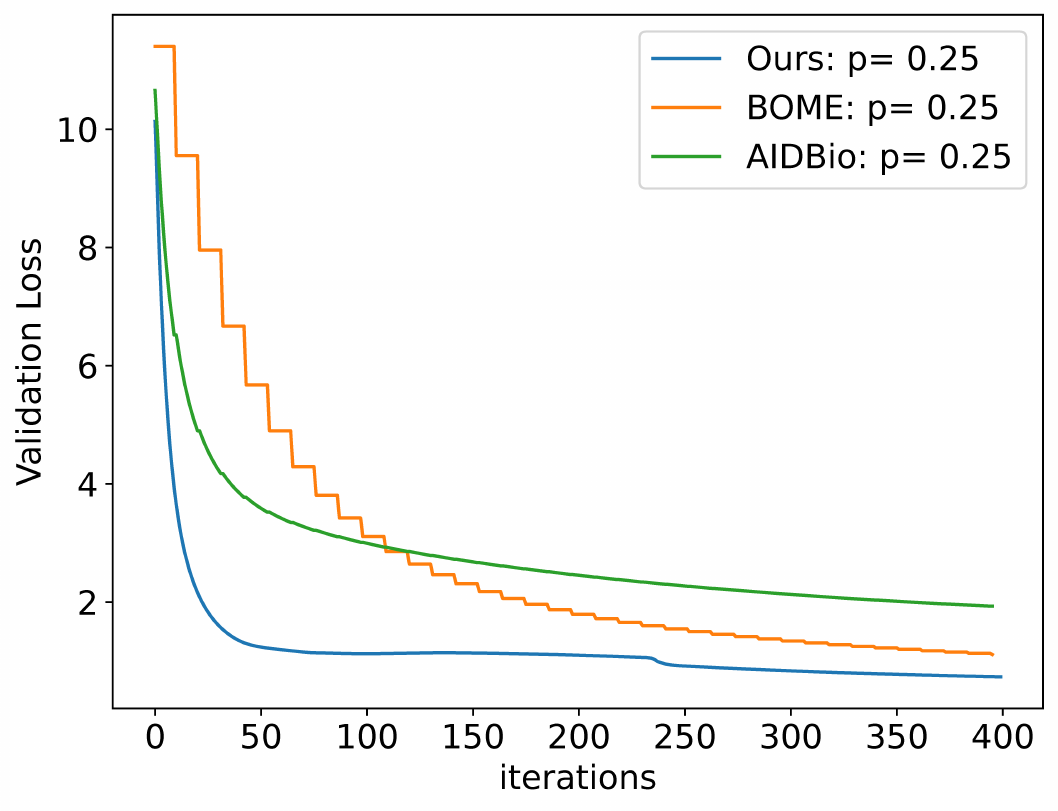}
        \includegraphics[width=0.24\linewidth]{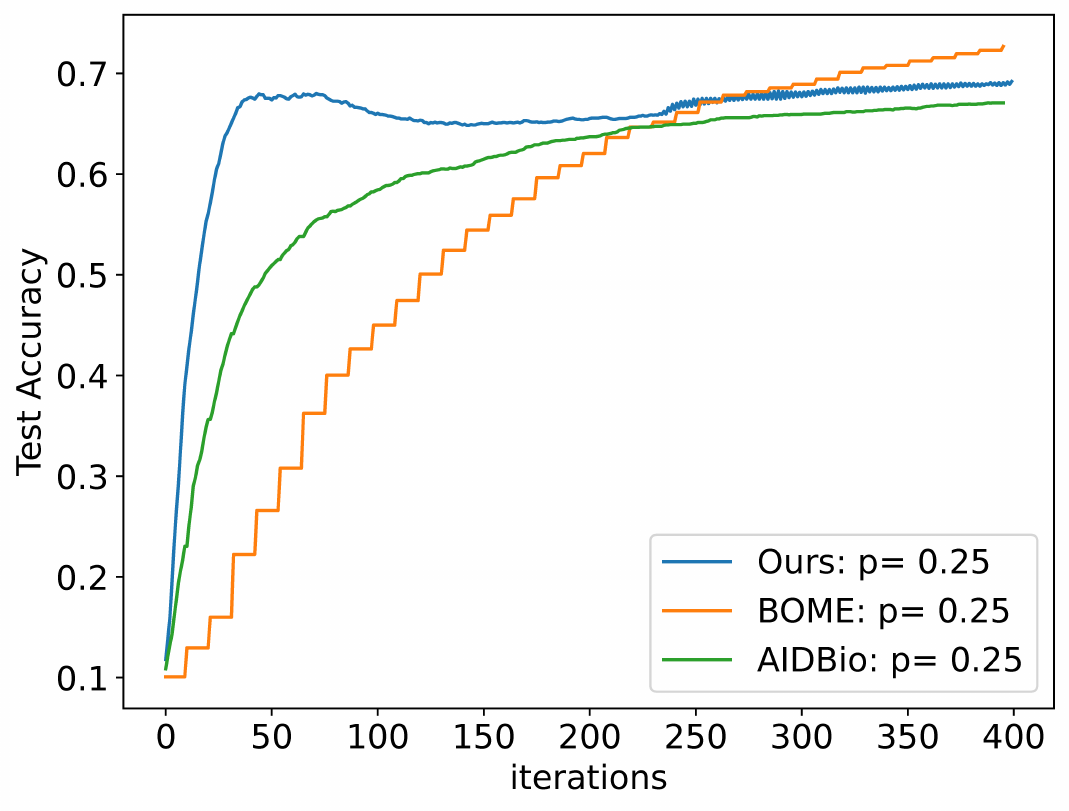}
        \includegraphics[width=0.24\linewidth]{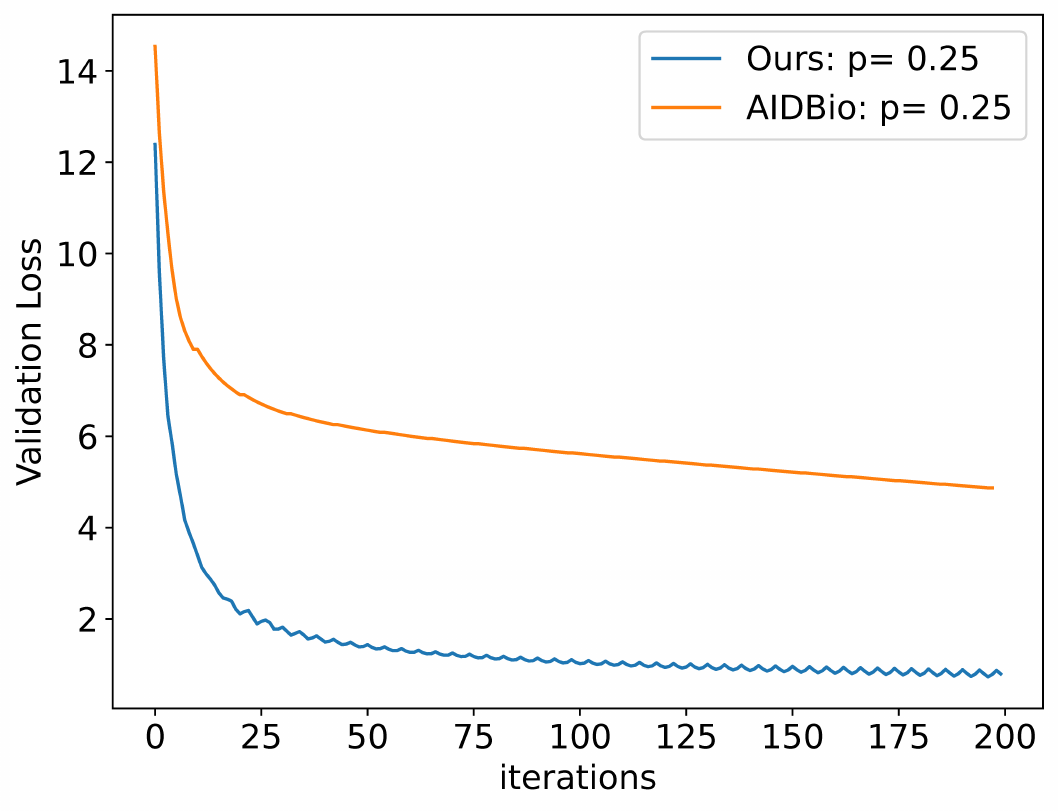}
        \includegraphics[width=0.24\linewidth]{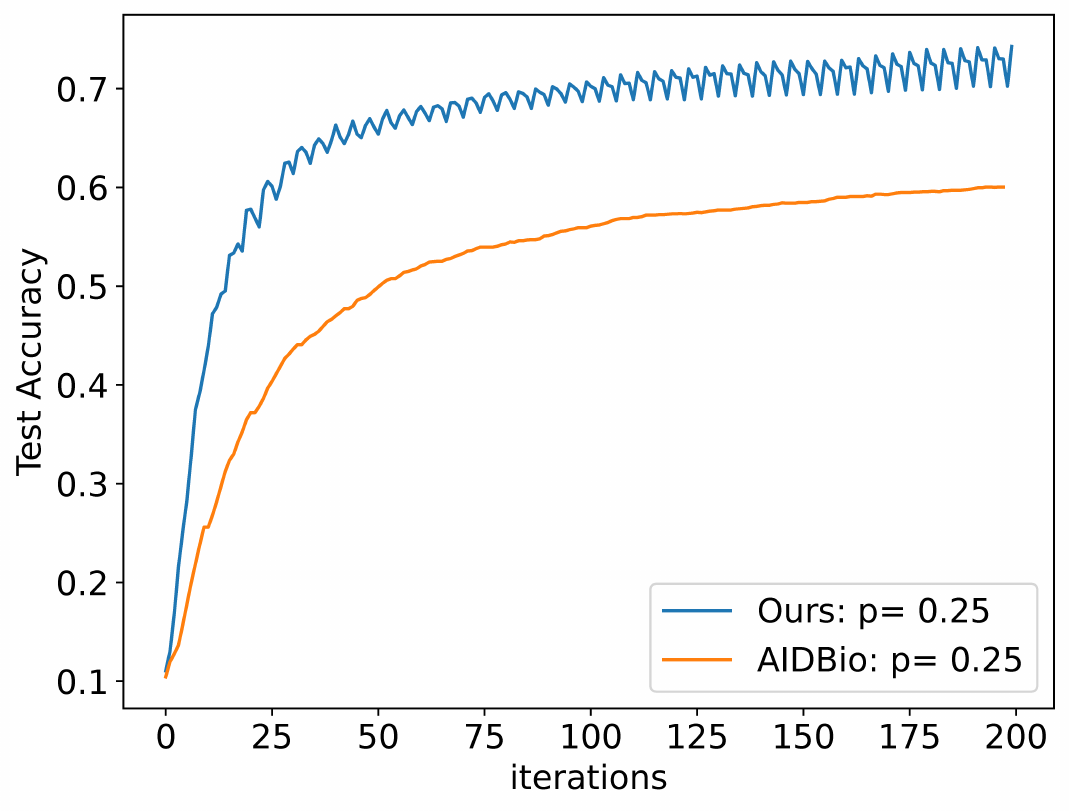}
        \caption{Comparison between our method with the state of the art on the DHC benchmark with corruption rate $p=25\%$.
        The first two plots show the validation loss and the accuracy of the test set on the DHC benchmark with PCA. 
        The last two plots show the validation loss and the accuracy of the test set on the large-scale DHC problem.}
        \label{fig:DHC}
     \end{figure*}
    
    \subsection{Data Hyper-Cleaning}
    Consider the DHC problem, where some of the labels in the training data have been corrupted, and the goal is to train a classifier utilizing the clean validation data. 
    The objective function is given by 
        \begin{align}
            &\min_x \frac{1}{N_{\text{val}}} \sum_{(a_{i}, b_{i}) \in \mathcal{D}_{\text{val}}} \mathcal{L}(a_{i}^\top y^\star(x), b_{i}) \notag \\
            &\text{s.t.} y^\star(x) \! = \! \argmin_y
            \frac{1}{N_{\text{tr}}} 
            \sum_{(a_{i}, b_{i}) \in \mathcal{D}_{\text{tr}}} \! \sigma(x_i) \mathcal{L}(a_{i}^\top y, b_{i}) \! + \! \lambda  \|y\|^2\!, \notag
        \end{align}
    where $\lambda=0.001$ is the regularizer and $\sigma(\cdot)$ and $\mathcal{L}(\cdot)$ represent the sigmoid function and cross-entropy loss, respectively.
    The upper-level variable $x \in \mathbb{R}^{5000}$ represents the sample weights, and lower-level variable $y$ is the weight of the classifier.
    We split the data into training $\mathcal{D}_{\text{tr}}$, validation $\mathcal{D}_{\text{val}}$, and test $\mathcal{D}_{\text{test}}$ and run the experiment under two setups, one reduces the dimensionality of the problem using Principle Component Analysis (PCA), and the other one tests our method in high-dimensional setting. 
    \subsubsection{Low-dimension DHC}
     We first use PCA to reduce the dimensions of the problem to $y \in \mathbb{R}^{82 \times 10}$ and run the experiment with corruption rate $p=25\%$.

     The first two plots in \Cref{fig:DHC} compare our method with BOME, and AIDBiO in terms of validation loss and test accuracy.

    \subsubsection{High-dimension DHC}
    In this experiment, we aim to study the performance of our method in high-dimensional benchmarks.
    Therefore, we do not use PCA, which translates to $y$ being a $784 \times 10$ matrix.

    The final two plots in \Cref{fig:DHC} compare our method with  AIDBiO in terms of validation loss and test accuracy

    \subsubsection{Neural Network Classifier}
  
    To evaluate our method on yet another large-scale nonconvex method, we use a fully connected neural network with one hidden layer and ReLU activation functions to solve the DHC problem. 

        \begin{align}
            &\min_x \frac{1}{N_{\text{val}}} \sum_{(a_{i}, b_{i}) \in \mathcal{D}_{\text{val}}} \mathcal{L}(f_{y^\star(x)} (a_{i}) , b_{i}) \notag \\
            &\text{s.t.} y^\star(x) \! = \! \argmin_y
            \frac{1}{N_{\text{tr}}} 
            \sum_{(a_{i}, b_{i}) \in \mathcal{D}_{\text{tr}}} \! \sigma(x_i) \mathcal{L}(f_{y} (a_{i}), b_{i}) \! + \! \lambda  \|y\|^2\!, \notag
        \end{align}
    where $f_\theta(.)$ denotes the neural network parameterized by $\theta$.
    
    \begin{figure}[t]
        \centering
        \includegraphics[width=0.39\linewidth]{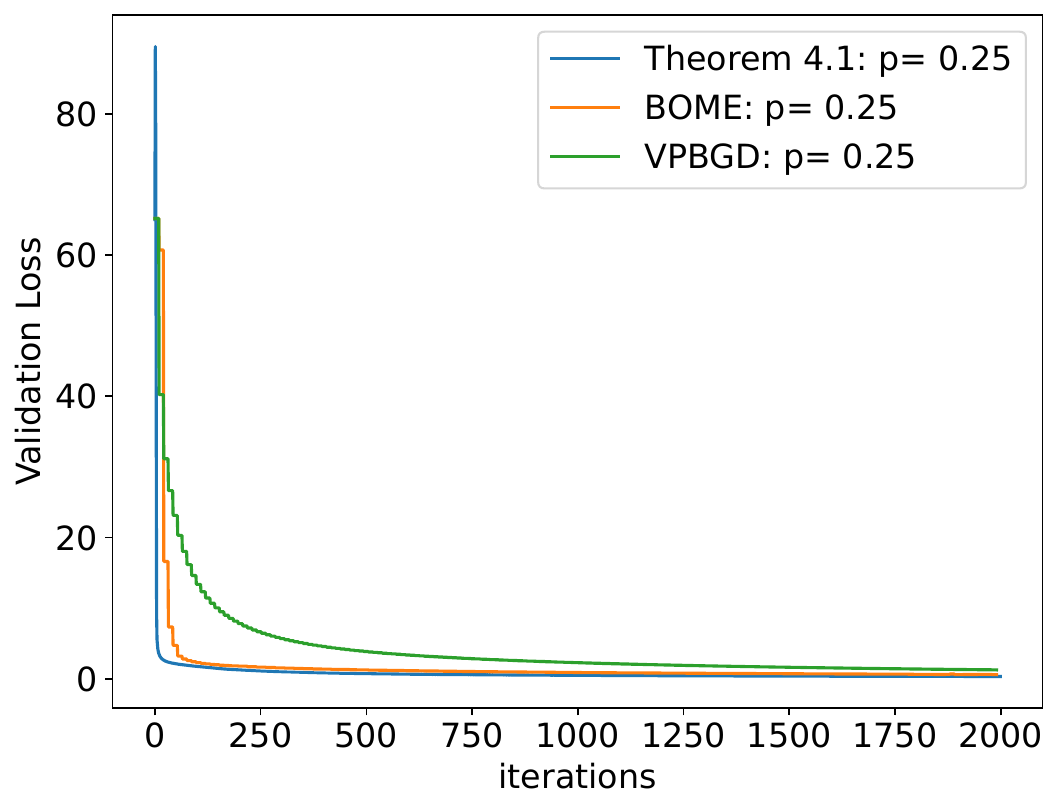}
        \includegraphics[width=0.39\linewidth]{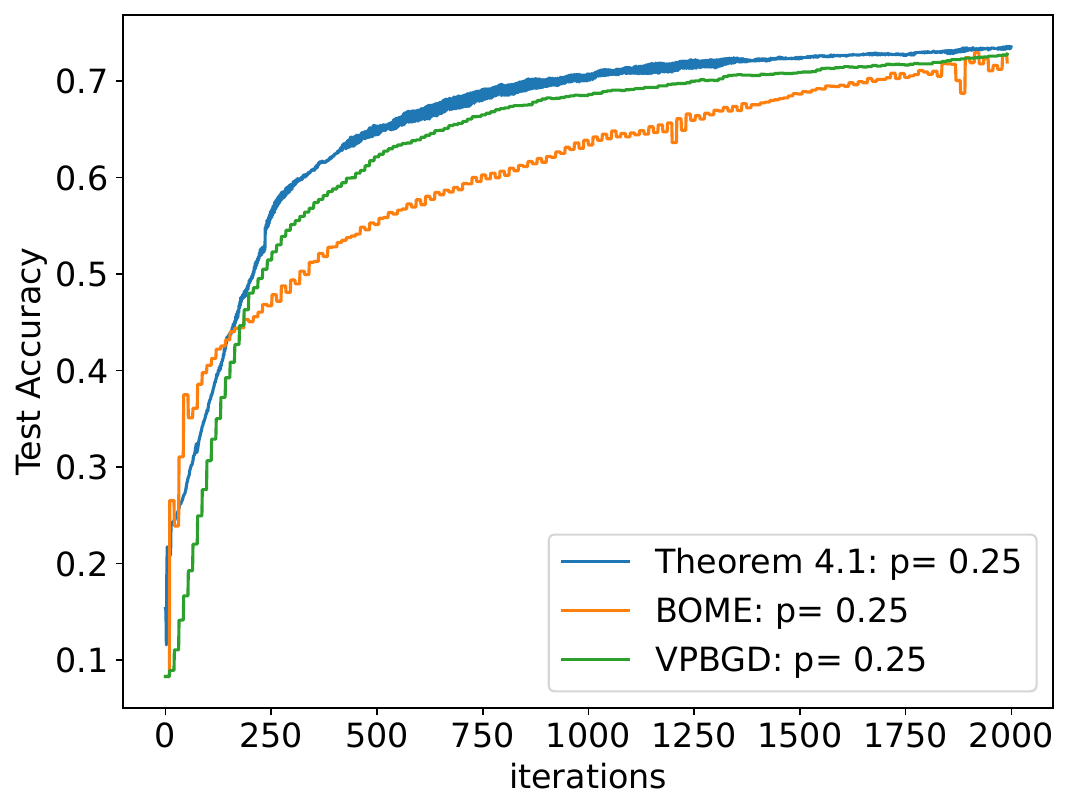}
        \caption{Comparisons of the validation loss and test accuracy between our method from \Cref{thm:conv_rate} with BOME and VPBGD on the DHC problem with neural network classifier.}
        \label{fig:NN}
    \end{figure}
    \Cref{fig:NN} shows that our method outperformed BOME and VPBGD in test accuracy. 
    The parameters for VPBGD were chosen from their git repository \footnote{ \href{https://github.com/hanshen95/penalized-bilevel-gradient-descent/}{https://github.com/hanshen95/penalized-bilevel-gradient-descent/}}.
    
    Furthermore, we did not compare with AIDBio due to the high computational burden of calculating the Hessians of neural networks.

% \vspace{-12mm}

\section{Conclusion}
In this paper, we proposed an inversion-free  {single-time scale} method to solve bilevel optimization problem of the form \eqref{eq: bilevel task 1}. 
Our idea hinged on the use of gradient descent to decrease the upper-level objective, coupled with a convex Quadratic Program (QP) that minimally perturbed the gradient descent directions to reduce the sub-optimality of the condition imposed by the lower-level problem. 
We proposed two methods, one assumed a certain regularity condition, and guaranteed to find a stationary point with an iteration complexity of $\mathcal{O}(1/\epsilon^{1.5})$, and the other one relaxed the assumption and in turn, proved a complexity of $\mathcal{O}(1/\epsilon^{3})$.
Furthermore, we ran extensive numerical analysis, showcasing the performance of our methods against state-of-the-art and under convex and non-convex settings.

% \newpage
% \section*{Impact Statement}
% This paper presents work that aims to advance the field of Machine Learning. 
% % We acknowledge numerous potential societal implications of our work, but we refrain from highlighting any specific ones in this context.
% We do not foresee any societal implications arising solely from our work.
% % \bibliography{refs}
% % \bibliographystyle{icml2025}
\printbibliography
% \bibliography{refs}

%%%%%%%%%%%%%%%%%%%%%%%%%%%%%%%%%%%%%%%%%%%%%%%%%%%%%%%%%%%%%%%%%%%%%%%%%%%%%%%
%%%%%%%%%%%%%%%%%%%%%%%%%%%%%%%%%%%%%%%%%%%%%%%%%%%%%%%%%%%%%%%%%%%%%%%%%%%%%%%
% APPENDIX
%%%%%%%%%%%%%%%%%%%%%%%%%%%%%%%%%%%%%%%%%%%%%%%%%%%%%%%%%%%%%%%%%%%%%%%%%%%%%%%
%%%%%%%%%%%%%%%%%%%%%%%%%%%%%%%%%%%%%%%%%%%%%%%%%%%%%%%%%%%%%%%%%%%%%%%%%%%%%%%
\clearpage
\appendix
\appendixpage

\section{Discussion on Related Works}\label{app:Related_works} 

\begin{remark}\label{rem:discussion}
    We make the following remarks regarding \Cref{tab:compare}:
    \begin{itemize}
    \item \textit{V-PBGD \cite{shen2023penalty}}: They establish conditions under which global or local minimizers of the penalized problem correspond to global or local minimizers of the original bilevel problem. However, the relation between the stationary points of the penalized problem with those of the original bilevel problem remains unsettled. In their setting,  a stationary point merely indicates that the gradient of the penalized objective vanishes, but this does not ensure that the penalty function $ p(x,y) :=g(x,y)-g(x)$ is zero. Consequently, such an stationary point does not satisfy the essential lower-level constraint $y\in\argmin g(x,\cdot)$ and therefore is not a valid solution to the original bilevel problem.
    \item \textit{GALET \cite{xiao2023generalized}}: In their work, the assumption \(\inf_{(x,y)}\{\sigma_{\min}^+(\nabla_{yy}^2 g(x,y))\} > 0\) is a strong global condition asserting that the Hessian of \(g\) w.r.t \(y\) is uniformly nondegenerate.  In other words, for every fixed \(x\), the function \(g(x,\cdot)\) exhibits a form of strong convexity which guarantees that the lower-level problem \(\min_y g(x,y)\) has a unique minimizer \(y^*(x)\) that depends smoothly on \(x\). Thus, this assumption along with
    PL condition, and Lipschitz continuity of Hessian  implies that the GALET assumes the strongly convex setting.  {For more details on this, see Appendix B of \cite{huang2024optimal}.}
    \item  \textit{HJFBiO \cite{huang2024optimal}}:
    For $\mu$-PL function g, they supposed that all singular values of $ \nabla_{yy}^2 g(x, y^*(x))$  lie in $[\mu, L_g]$  implying the Hessian is positive definite at the minimizer \( y^*(x) \), ensuring local strong convexity of \ $g(x, y)$  around $y^*(x)$. This is stronger than the PL inequality (which only ensures gradient growth) but weaker than strong convexity (which requires positive definiteness everywhere). The assumption guarantees that \( y^*(x) \) is an   {isolated} local minimizer, enables smooth dependence of \( y^*(x) \) on \( x \). 
        Moreover, as discussed in \cite{xiao2023generalized}, the algorithms in \cite{huang2024optimal} explicitly use Hessian oracles.
    \end{itemize}
\end{remark}

\section{Bilevel Optimization Application}
\textbf{Meta Learning:} %\todo{What makes it few-shot learning? Maybe just say meta-learning. Also change notation $k$ to $t$.} 
Consider a few-shot meta-learning problem with $t$ tasks, $\{ \mathcal{T}_i \}_{i=1}^{t}$, sampled from a task distribution. Each task $\mathcal{T}_i$ is associated with a loss function $\ell(x, y_i; \xi)$ over data $\xi$, where $x \in \mathbb{R}^n$ represents the parameters of a shared embedding model across all tasks, and $y_i \in \mathbb{R}^m$ denotes the task-specific parameters. The objective is to determine an optimal $x$ that generalizes across all tasks while simultaneously learning the task-specific parameters $\{ y_i \}_{i=1}^{t}$ by minimizing their respective losses. In the lower-level problem, each task learner $\mathcal{T}_i$ optimizes $y_i$ by minimizing its loss over a training dataset $\mathcal{D}^{i}_{\text{tr}}$. In the upper-level problem, the meta-learner searches for a globally shared parameter $x$ that performs well across all tasks over a test dataset $\mathcal{D}_{\text{te}} = \{ \mathcal{D}^{1}_{\text{te}}, \dots, \mathcal{D}^{t}_{\text{te}} \}$, where the minimizers $\{ y_i \}_{i=1}^{t}$ from the lower-level problem are used \cite{huang2023momentum}. Let $\tilde{y} = (y_1, \dots, y_t)$ to denote the collection of all task-specific parameters, then the bilevel formulation can be represented as 
\begin{align*}
    \min_{x} ~ & \ell_{\mathcal{D}_{\text{te}}}(x, \tilde{y}^*) = \frac{1}{t} \sum_{i=1}^{t} \mathbb{E}_{\xi \sim \mathcal{D}^{i}_{\text{te}}} \left[ \ell_{\text{te}}(x, y_i^*; \xi) \right]
    \\
    \text{s.t.} ~ & y^* \in \arg\min_{\tilde{y}} \ell_{\mathcal{D}_{\text{tr}}}(x, \tilde{y}) = \frac{1}{t} \sum_{i=1}^{t} \mathbb{E}_{\xi \sim \mathcal{D}^{i}_{\text{tr}}} \left[ \ell_{\text{tr}}(x, y_i; \xi) \right],
\end{align*}
where, $\mathcal{D}^{i}_{\text{tr}}$ and $\mathcal{D}^{i}_{\text{te}}$ represent the training and test datasets for task $\mathcal{T}_i$, respectively. The lower-level problem optimizes each task-specific parameter $y_i$ by minimizing its corresponding loss $\ell_{\text{tr}}(x, y_i; \xi)$.

\section{Proofs}
\begin{lemma}\label{lem:lip_h}
    Suppose Assumption \ref{assumption:lowerlevel} hold. Then, function \( h: \mathbb{R}^n \times \mathbb{R}^m \to \mathbb{R} \) defined by $h(x, y) = \|\nabla_y g(x, y)\|^2$
    is Lipschitz continuous with constant \( L_h = 2 C_g (L_{yy}^g + L_{yx}^g ) \).
\end{lemma}
% \begin{proof}
%     See \Cref{proof:lem_lip_h}.
% \end{proof}
% \section{Proof of  \cref{lem:lip_h}:}\label{proof:lem_lip_h}
\begin{proof}
    Consider any two points \( (x_1, y_1) \) and \( (x_2, y_2) \) in \( \mathbb{R}^n \times \mathbb{R}^m \). We have:
\[
h(x_1, y_1) - h(x_2, y_2) = \|\nabla_y g(x_1, y_1)\|^2 - \|\nabla_y g(x_2, y_2)\|^2.
\]
This can be rewritten using the identity \( a^2 - b^2 = (a + b)(a - b) \) followed by utilizing Assumption \ref{assumption:lowerlevel} along with the application of triangle inequality as 
\begin{align*}
    |h(x_1, y_1) - h(x_2, y_2)| &\leq 2 C_y^g \| \nabla_y g(x_1, y_1) - \nabla_y g(x_2, y_2)\| \\
    &\leq 2 C_y^g \| \nabla_y g(x_1, y_1) - \nabla_y g(x_1, y_2)+\nabla_y g(x_1, y_2)- \nabla_y g(x_2, y_2)\| \\
    & \leq 2 C_y^g \big( L_{yy}^g \| y_1 -y_2\| + L_{yx}^g \|x_1 - x_2\|\big) \\
    & \leq 2 C_y^g ( L_{yy}^g + L_{yx}^g) \| (x_1, y_1) - (x_2, y_2)\|
\end{align*}
which implies the result.
    \end{proof}
For the sake of simplicity throughout the remainder of the proofs, we define the vectors \( z_k = (x_k, y_k) \) and \( \Delta_k = (\Delta^x_k, \Delta^y_k) \), both belonging to \( \mathbb{R}^{n+m} \).
% \red{Remove $\delta$ from the proof}
\subsection[Proof of Theorem~\ref{thm:conv_rate}]%
           {Proof of \cref{thm:conv_rate}}\label{proof:thm_conv_rate}

\begin{proof}
    Function $f$ is continuously differentiable with a Lipschitz continuous gradient, characterized by  \( L_f = \max\{L_x^f, L_y^f\} \). Using the smoothness of the function $f$, we have 
    \begin{align}\label{eq:desc}
    f(z_{k+1}) &\leq f(z_k) + \nabla f(z_k)^\top (z_{k+1} - z_k) +\frac{L_f}{2}\|z_{k+1}-z_k\|^2\nonumber \\
    &= f(z_k) +\gamma \nabla_z f(z_k)^\top \Delta_k + \frac{\gamma^2 L_f}{2} \|\Delta_k \|^2 \nonumber \\
    &=f(z_k) +\gamma (\nabla f(z_k)+\Delta_k)^\top \Delta_k + (\frac{\gamma^2 L_f}{2}-\gamma) \|\Delta_k \|^2 \nonumber \\
    &=f(z_k) -\gamma \lambda(z_k) \nabla h(z_k)^\top \Delta_k + (\frac{\gamma^2 L_f}{2}-\gamma) \|\Delta_k \|^2 .
\end{align}
% where the last inequality obtained using the constraint of \eqref{eq: first-order lower-level dynamics 2}.
Since $(\Delta_k,\lambda(z_k))$ is the optimal primal-dual pair for the subproblem in \eqref{eq: first-order lower-level discrete time2} at iteration $k$, the complementarity slackness condition implies that $\lambda(z_k)(\nabla h(z_k)^\top \Delta_k+\alpha \|\nabla h(z_k)\|^2)=0$. Using this result within the above inequality we obtain
\begin{align}\label{eq:upper-bound-f}
    f(z_{k+1})-f(z_k) 
    % &\leq (\frac{\gamma^2 L_f}{2}-\gamma) \|\Delta_k \|^2 +\gamma \alpha \lambda(z_k) (\|\nabla h(z_k)\|^2-\delta^2)\nonumber \\
    \leq (\frac{\gamma^2 L_f}{2}-\gamma) \|\Delta_k \|^2 +\gamma \alpha \lambda(z_k)\|\nabla h(z_k)\|^2.
\end{align}

Similarly, using the smoothness of function $h$ and the update of $\Delta_k$, we have 
\begin{align}\label{eq:smoothness-h}
    h(z_{k+1})-h(z_k)\nonumber 
    &\leq \gamma\nabla h(z_k)^\top\Delta_k + \frac{\gamma^2 L_h}{2} \|\Delta_k \|^2  \nonumber\\
    &=-\gamma \nabla h(z_k)^\top \nabla f(z_k)-\gamma\lambda(z_k)\|\nabla h(z_k)\|^2 +  \frac{\gamma^2 L_h}{2} \|\Delta_k \|^2\nonumber\\
    &\leq \frac{\gamma}{2\alpha L_h}\|\nabla h(z_k)\|^2+\frac{\alpha \gamma L_h}{2}C_f^2-\gamma\lambda(z_k)\|\nabla h(z_k)\|^2 +  \frac{\gamma^2 L_h}{2} \|\Delta_k \|^2,
\end{align}
where in the last inequality we used Young's inequality 
$$-\nabla f(z_k)^T \nabla h(z_k) \leq  \frac{\alpha L_h}{2} \|\nabla f(z_k)\|_2^2 + \frac{1}{2\alpha L_h} \|\nabla h(z_k)\|_2^2 $$ 
as well as boundedness of $\nabla f$. Let us define $v(z)\triangleq f(z)+\alpha h(z)$.
Combining the above inequalities by multiplying \eqref{eq:smoothness-h} with $\alpha$ and adding to \eqref{eq:upper-bound-f} we obtain  
\begin{align}
    v(z_{k+1})-v(z_k)\nonumber 
    & \leq  \frac{\gamma}{2 L_h}\|\nabla h(z_k)\|^2+\frac{\alpha^2 \gamma L_h}{2}C_f^2+ \gamma(\frac{ L_f+\alpha L_h}{2}\gamma -1) \|\Delta_k \|^2.
\end{align}
Next, assuming that $\gamma\leq \frac{1}{L_f+\alpha L_h}$ and rearranging the terms we conclude that 
\begin{align}\label{eq:bound-Deltak}
    \frac{1}{2}\|\Delta_k \|^2\leq \frac{v(z_k)-v(z_{k+1})}{\gamma}+  \frac{1}{2 L_h}\|\nabla h(z_k)\|^2+\frac{\alpha^2  L_h}{2}C_f^2.
\end{align}
On the other hand, using the smoothness of $h$ once again along with the fact that $\Delta_k$ is a feasible point of \eqref{eq: first-order lower-level discrete time2} we have 

\begin{align*}
    h(z_{k+1})-h(z_k)&\leq \gamma \nabla h(z_k)^\top \Delta_k +\frac{\gamma^2 L_h}{2}\|\Delta_k\|^2\nonumber\\
    &\leq -\gamma \alpha(\|\nabla h(z_k)\|^2)+\frac{\gamma^2 L_h}{2}\|\Delta_k\|^2.
\end{align*}
Rearranging the terms leads to
\begin{align}\label{eq:bound-norm-h}
    \frac{1}{2L_h}\|\nabla h(z_k)\|^2&\leq \frac{h(z_k)-h(z_{k+1})}{2\alpha \gamma L_h}+\frac{\gamma}{4\alpha }\|\Delta_k\|^2.
\end{align}
Now adding up \eqref{eq:bound-Deltak} and \eqref{eq:bound-norm-h} and using $\gamma\leq \alpha$ leads to 
\begin{align*}
    \frac{1}{4}\|\Delta_k \|^2 &\leq \frac{v(z_k)-v(z_{k+1})}{\gamma}+\frac{h(z_k)-h(z_{k+1})}{2\alpha \gamma L_h} +  \frac{\alpha^2  L_h}{2}C_f^2
    % +\frac{\delta^2}{2L_h}.
\end{align*}
Summing the above inequality over $k=0$ to $K-1$ and divide both sides $K/4$ leads to
\begin{align}\label{eq:rate-Deltak}
    \frac{1}{K}\sum_{k=0}^{K-1}\|\Delta_k \|^2&\leq \frac{4(v(z_0)-v(z_{K}))}{\gamma K}+\frac{2(h(z_0)-h(z_{K}))}{\alpha \gamma L_h K}  +  2{\alpha^2  L_h}C_f^2\nonumber\\
    & \leq \frac{4(f(z_0)+\alpha^3 C_0-\bar{f})}{\gamma K}+\frac{2\alpha C_0}{\gamma L_h K} +  2{\alpha^2  L_h}C_f^2.
\end{align}
where in the last inequality we used nonegativity of function $h$, the lower-bound on function $f$, and the initialization condition $h(z_0)\leq \alpha^2 C_0$. 

Furthermore, summing the result in \eqref{eq:bound-norm-h} over $k=0$ to $K-1$ and dividing both sides by $K/(2L_h)$ and using \eqref{eq:rate-Deltak} implies that
\begin{align}\label{eq:rate-norm-h}
    \frac{1}{K}\sum_{k=0}^{K-1}\|\nabla h(z_k)\|^2&\leq \frac{h(z_0)-h(z_K)}{\alpha \gamma K}+\frac{L_h\gamma}{2\alpha K}\sum_{k=0}^{K-1}\|\Delta_k\|^2\nonumber\\
    &\leq \frac{2\alpha C_0}{\gamma K}+\frac{2L_h(f(z_0)+\alpha^3 C_0-\bar{f})}{\gamma K} +  {\alpha^2  L_h^2}C_f^2. 
\end{align}
\end{proof}

\subsection{Proof of Corollary~\ref{cor:rate-nabla-h}:}\label{proof:corollary_rate_nabla_h}
\begin{proof}
    Suppose $\rho(x,y)=\|\nabla h(x,y)\|^2$, using the result of \cref{thm:conv_rate}, we have the following convergence bounds
    \begin{align*}
    \frac{1}{K}\sum_{k=0}^{K-1} \left( \|\Delta_k^x \|^2 +\|\Delta_k^y \|^2 \right)
    &\leq \frac{4(f_0+C_0-\bar{f})}{\gamma K} + \frac{2C_0}{\gamma L_h K} + 2{\alpha^2  L_h}C_f^2,
\end{align*}
and, 
\begin{align*}
    \frac{1}{K}\sum_{k=0}^{K-1} \|\nabla h(x_k, y_k)\|^2  
    &\leq \frac{2C_0}{\gamma K} + {\alpha^2  L_h^2}C_f^2 + \frac{2L_h(f_0+C_0-\bar{f})}{\gamma K}.
\end{align*}

Setting $\alpha = K^{-1/3}$ and $\gamma = \min\{\alpha, \frac{1}{L_f + \alpha L_h}\}$, implies that $\gamma = \Omega( 1/K^{1/3})$. Substituting $\alpha$ and $\gamma$ we obtain $\frac{1}{K} \sum_{k=0}^{K-1} ( \|\Delta_k^x \|^2 +\|\Delta_k^y \|^2 ) = \mathcal{O}(K^{-2/3}),$ and $\frac{1}{K} \sum_{k=0}^{K-1} \|\nabla h(x_k, y_k)\|^2 = \mathcal{O}(K^{-2/3})$. Now let $t\triangleq \argmin_{0\leq k\leq K-1}\{  {\max}\{\|\Delta_k\|^2,\|\nabla h(x_k,y_k)\|^2\}\}$, then we conclude that
%There must exist at least one iterate $t \in \{0, 1, \dots, K-1\}$ for which both $ (\|\Delta_k^x \|^2 +\|\Delta_k^y \|^2 )$ and $\|\nabla h(x_t, y_t)\|^2$ are individually bounded by $\mathcal{O}(K^{-2/3})$. By selecting $K = \mathcal{O}(1/\epsilon^{1.5})$, these bounds ensure that 
$ (\|\Delta_t^x \|^2 +\|\Delta_t^y \|^2 ) \leq \epsilon$ and $\|\nabla h(x_t, y_t)\|^2 \leq \epsilon$ after $K = \mathcal{O}(1/\epsilon^{1.5})$ iterations.
%and consequently the algorithm achieves the desired accuracy within $\mathcal{O}(1/\epsilon^{1.5})$ iterations.
\end{proof}

\subsection{Proof of \cref{cor:rate_nabla_g}:}\label{proof:corollary_rate_nabla_g}
\begin{proof}
    Corollary \ref{cor:rate-nabla-h} ensures that 
    $\frac{1}{K} \sum_{k=0}^{K-1} ( \|\Delta_k^x \|^2 +\|\Delta_k^y \|^2 ) = \mathcal{O}(K^{-2/3}),$ and $\frac{1}{K} \sum_{k=0}^{K-1} \|\nabla h(x_k, y_k)\|^2 = \mathcal{O}(K^{-2/3})$.
    %the average of \(\|\nabla h(x_k,y_k)\|^2\) and $(\|\Delta_k^x \|^2 +\|\Delta_k^y \|^2 )$ over the first \(K\) iterations is bounded by \(\mathcal{O}(K^{-2/3})\), and there exists an iterate \(t \in \{0, 1, \dots, K-1\}\) such that both \(\|\nabla h(x_t, y_t)\|^2\) and \(\|\Delta_t\|^2 = \|\nabla f(x_t, y_t) + \lambda_t \nabla h(x_t, y_t)\|^2\) are individually bounded by \(\mathcal{O}(K^{-2/3})\). 
    Utilizing the regularity condition, it follows that \(\|\nabla_y g(x_k, y_k)\|^2 \leq c^2\,\|\nabla h(x_k, y_k)\|^2\), hence, $\frac{1}{K} \sum_{k=0}^{K-1} \|\nabla_y g(x_k, y_k)\|^2 = \mathcal{O}(K^{-2/3})$.
    Now let $t\triangleq \argmin_{0\leq k\leq K-1}\{  {\max}\{\|\Delta_k\|^2,\|\nabla_y g(x_k,y_k)\|^2\}\}$, then we conclude that
    \(\|\nabla_y g(x_t, y_t)\|^2 \leq \epsilon\) and \(\|\nabla f(x_t, y_t) + \lambda_t \nabla h(x_t, y_t)\|^2 \leq \epsilon\) within $K=\mathcal O(1/\epsilon^{1.5})$ iterations, thereby guaranteeing that \((x_t, y_t, \lambda_t)\) is an \(\epsilon\)-KKT point of problem \eqref{eq: bilevel task 2 a}.
\end{proof}
\subsection{Proof of \cref{thm:conv_rate2}:}\label{proof:thm_conv_rate_2}
Using the smoothness of $h$ once again along with the fact that $\Delta_k$ is a feasible point of \eqref{eq: first-order lower-level discrete time2} we have 
\begin{align}\label{eq:h-smooth}
    h(z_{k+1})-h(z_k)&\leq \gamma \nabla h(z_k)^\top \Delta_k +\frac{\gamma^2 L_h}{2}\|\Delta_k\|^2\nonumber\\
    &\leq -\gamma \alpha(\|\nabla h(z_k)\| {(h(z_0))^{1/2})}+\frac{\gamma^2 L_h}{2}\|\Delta_k\|^2.
\end{align}
Note that $\Delta_k$ can bounded as follows: 
\begin{equation}
    \|\Delta_k\|\leq \|\nabla f(z_k)\|+\lambda(z_k)\|\nabla h(z_k)\|\leq B_\Delta\triangleq 2C_f+\alpha^2  {\sqrt{C_0}}.
\end{equation}
Then, from \eqref{eq:h-smooth} and using the above bound followed by a telescopic summation we have that for any $k\geq 0$
\begin{align*}
    h(z_k)\leq h(z_0)+\frac{\gamma^2L_h}{2}B_\Delta^2 k.
\end{align*}
Taking the average from the above inequality over $k=0$ to $K-1$ and using the initialization condition we obtain
\begin{align}\label{eq:rate-hk}
    \frac{1}{K}\sum_{k=0}^{K-1} h(z_k)\leq \alpha^2 C_0+\frac{\gamma^2L_h B_\Delta^2}{2}\frac{(K-1)}{2}.
\end{align}
Moreover, similar to the proof of the previous result, we can show that
\begin{align}\label{eq:upper-bound-f2}
    f(z_{k+1})-f(z_k)&\leq (\frac{\gamma^2 L_f}{2}-\gamma) \|\Delta_k \|^2 +\gamma \alpha \lambda(z_k)\|\nabla h(z_k)\| {(h(z_0))^{1/2}}\nonumber\\
    &\leq (\frac{\gamma^2 L_f}{2}-\gamma) \|\Delta_k \|^2 +\gamma \alpha B_\Delta  {(h(z_0))^{1/2}}\nonumber\\
    &\leq (\frac{\gamma^2 L_f}{2}-\gamma) \|\Delta_k \|^2 +\frac{\gamma \alpha^2}{2} B_\Delta^2 +\frac{\gamma}{2}  {h(z_0)}.
\end{align}
Now, selecting $\gamma\leq 1/L_f$, rearranging the terms and taking average over $k=0$ to $K-1$ we obtain
\begin{align}
    \frac{1}{K}\sum_{k=0}^{K-1}\|\Delta_k\|^2&\leq \frac{2(f_0-f_K)}{\gamma K}+\alpha^2 B_\Delta^2 + \frac{1}{K}\sum_{k=0}^{K-1} {h(z_0)}\nonumber\\
    &\leq \frac{2(f_0-f_K)}{\gamma K}+\alpha^2 B_\Delta^2 + \alpha^2 C_0 %+\frac{\gamma^2L_h B_\Delta^2}{2}\frac{(K-1)}{2}.
\end{align}

\subsection{Proof of \cref{cor:rate_nabla_g2}:}\label{proof:corollary_rate_nabla_g2}
\begin{proof}
    Suppose $\rho(x,y) = \| \nabla h(x,y)\|\sqrt{h(x_0,y_0)}$, using the result of \cref{thm:conv_rate2}, we have the following convergence bound
     \[
        \frac{1}{K} \sum_{k=0}^{K-1} \|\Delta_k\|^2 \leq \frac{2(f_0 - f_K)}{\gamma K} + \alpha^2 B_\Delta^2 + \alpha^2 C_0 %+ \frac{\gamma^2 L_h B_\Delta^2 (K-1)}{4},
    \]
    and,
    \[
        \frac{1}{K} \sum_{k=0}^{K-1} h(x_k, y_k) \leq \alpha^2 C_0 + \frac{\gamma^2 L_h B_\Delta^2 (K-1)}{4}.
    \]
  Setting $\alpha = K^{-1/6}$ and $\gamma =  \min\{\frac{1}{K^{2/3}}, \frac{1}{L_f}\}$ implies that $\gamma = \Omega(1/{K^{2/3}})$.  Substituting $\alpha$ and $\gamma$ into these inequalities simplifies them to $ \frac{1}{K} \sum_{k=0}^{K-1} \|\Delta_k\|^2 = \mathcal{O}({K^{-1/3}})$ and $\frac{1}{K} \sum_{k=0}^{K-1} h(x_k, y_k) = \mathcal{O}({K^{-1/3}})$. Now let 
  $t\triangleq \argmin_{0\leq k\leq K-1}\{\max\{\|\Delta_k\|^2,h(x_k,y_k)\}\}$, then we conclude that $\|\Delta_t\|^2 \leq \mathcal{O}({K^{-1/3}})$ and $h(x_t, y_t) \leq \mathcal{O}({K^{-1/3}})$. Utilizing the definition of $h(x,y) =\|\nabla_y g(x,y)\|^2$, we directly obtain $\|\nabla_y g(x_t, y_t)\|^2 \leq \mathcal{O}(1/K^{1/3})$. To achieve $\|\nabla_y g(x_t, y_t)\|^2 \leq \epsilon$ and $\|\nabla f(x_t, y_t) + \lambda_t \nabla h(x_t, y_t)\|^2 \leq \epsilon$, we require that $1/K^{1/3} \leq \epsilon$, which leads to choosing $K = \mathcal{O}(1/\epsilon^{3})$. Consequently, our proposed algorithm can achieve an $\epsilon$-KKT point $(x_t, y_t, \lambda_t)$ within $\mathcal{O}(1/\epsilon^{3})$ iterations.

\end{proof}

\section{Additional Experiments}\label{appendix:exp}
    \subsection{Coreset Selection}
    Following \cite{liu2022bome}, we consider the following coreset selection problem, which is a bi-level optimization problem with a strongly convex lower-level function, 
    \begin{align}
        &\min_x \quad \|y^\star(x) - y_0\|_2^2 \notag \\
        &\text{s.t.} \quad y^\star (x) \in \argmin_y \|y - A \sigma(x)\|_2^2 \notag
    \end{align}
    where $\sigma(x) = \exp(x) / \sum_{i=1}^4 \exp(x_i)$ is the softmax function, $x \in \mathbb{R}^4$,$y \in \mathbb{R}^2$ and $A \in \mathbb{R}^{2 \times 4}$. We compare our method against BOME \cite{liu2022bome} and AIDBiO \cite{ji2021bilevel}.
     {
    The result can be viewed in \Cref{fig:Toy example Coreset}, which shows that our method converges much faster than BOME and AIDBiO.
    }
    % Note that the convergence of BOME requires Polyak-{\L}ojasiewicz condition for the lower-level function and AIDBiO requires the lower-level function to be strongly-convex. 
    % Note that AIDBiO converges much slower than both BOME and our method, a potential reason for this is that the convergence of AIDBiO requires the lower-level function to be strongly-convex, which is apparently not true in this example.

    \begin{figure}
        \centering
        \includegraphics[width=0.33\linewidth]{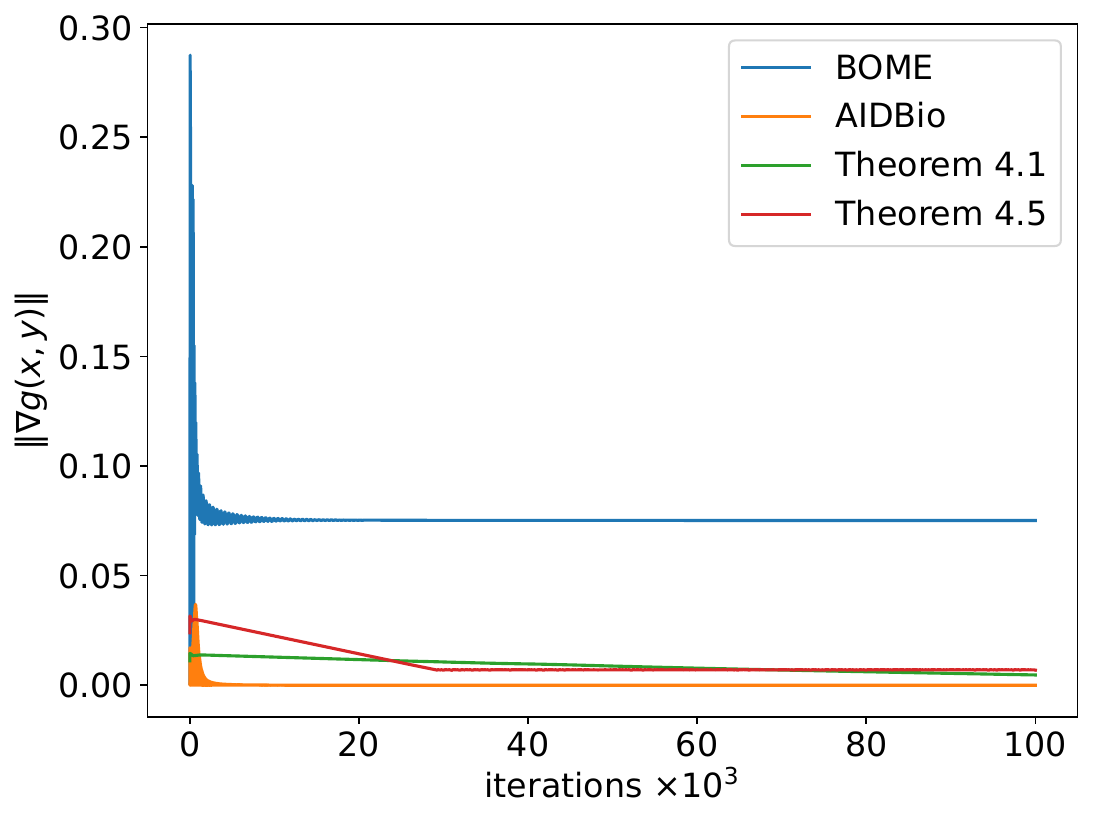}
        \includegraphics[width=0.33\linewidth]{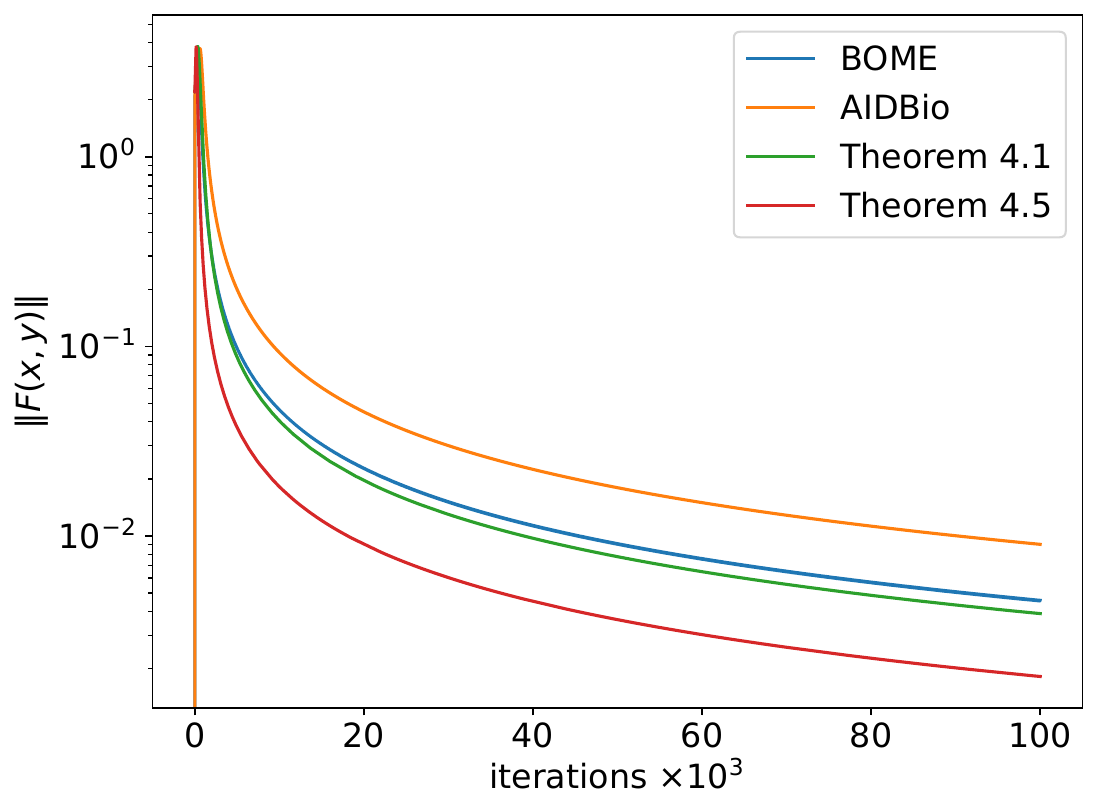}
        \caption{Comparison of our method, AIDBiO, and BOME on the coreset selection problem.}
        \label{fig:Toy example Coreset}
    \end{figure}

\end{document}